\documentclass[11pt]{amsart}
\usepackage{amsfonts}
\usepackage{amsmath}
\usepackage{amssymb}
\usepackage{multicol}
\usepackage{stmaryrd}
\usepackage{cite}
\usepackage{epsfig}
\usepackage{color}
\usepackage{graphics}
\usepackage{graphicx}
\usepackage{epstopdf}
\usepackage{multicol,graphics}
\newcommand\bes{\begin{eqnarray}}
\newcommand\ees{\end{eqnarray}}

\newtheorem{theorem}{Theorem}[section]
\newtheorem{lemma}[theorem]{Lemma}
\newtheorem{corollary}[theorem]{Corollary}

\newtheorem{definition}[theorem]{Definition}
\newtheorem{remark}[theorem]{Remark}
\newtheorem{proposition}[theorem]{Proposition}

\numberwithin{equation}{section}
\allowdisplaybreaks

\begin{document}

\title[Propagation Phenomena for Nonlocal Dispersal Equations]{\textbf{Propagation Phenomena for Nonlocal Dispersal Equations in Exterior Domains}}

\author[Qiao, Li and Sun]{Shao-Xia Qiao, Wan-Tong Li$^{*}$~ and Jian-Wen Sun}
\thanks{\hspace{-.6cm}
School of Mathematics and Statistics, Lanzhou University, Lanzhou, Gansu, 730000, P.R. China.\\
$^*${\sf Corresponding author} (wtli@lzu.edu.cn)\\
$^+${\sf Submitted to JDDE on August 7, 2019}}
\date{\today}

\begin{abstract}

This paper is concerned with the spatial propagation of nonlocal dispersal equations with bistable or multistable nonlinearity in exterior domains. We obtain the existence and uniqueness of an entire solution which behaves like a planar wave front as time goes to negative infinity. In particular, some disturbances on the profile of the entire solution happen as the entire solution comes to the interior domain. But the disturbances disappear as the entire solution is far away from the interior domain. Furthermore, we prove that the solution can gradually recover its planar wave profile and continue to propagate in the same direction as time goes to positive infinity for compact convex interior domain. Our work generalizes the local (Laplace) diffusion results obtained by Berestycki et al. (2009) to the nonlocal dispersal setting by using new known Liouville results and Lipschitz continuity of entire solutions due to Li et al. (2010).


\end{abstract}

\maketitle

\section{Introduction }
\noindent

This paper is concerned with the nonlocal dispersal problem
\begin{equation}\label{1-1}
u_t(x,t)=\int_{\mathbb{R}^N\backslash K}J(x-y)[u(y,t)-u(x,t)]dy+f(u),~ x\in\Omega=\mathbb{R}^N\backslash K,
\end{equation}
where $K$ is a compact subset of $\mathbb{R}^N$ and $f$ is a bistable or multistable type nonlinearity. Throughout the paper, we assume that $f$ and $J$ satisfy the following assumptions.
\begin{itemize}
\item[(F)] $f\in C^{1,1}([0,1])$ such that
\begin{equation*}
f(0)=f(1)=0,~f'(0)<0, ~f'(1)<0,~f'(s)<\inf\limits_{x\in\Omega}\int_{\Omega}J(x-y)dy<1,~\text{and}~\int_0^1f(s)ds>0.
\end{equation*}
\item[(J)] The kernel function $J\in C^1(\mathbb{R}^N)$ is radial symmetry and compactly supported such that
    \[J(x)\geq0~\text{for}~x\in\mathbb{R}^N,~J(0)>0~\text{and}~\int_{\mathbb{R}^N}J(y)dy=1.\]
\end{itemize}

It is well-known that the local dispersal problem in exterior domain is well established by Berestycki et al. \cite{BHM}. In order to study how a planar wave front propagates around an obstacle, they considered the following semi-linear
parabolic problem
\begin{equation}\label{1-2}\left\{
\begin{aligned}
&u_t=\Delta u+f(u),~ x\in\Omega,\\
&\nu\cdot\nabla u=0,~x\in\partial\Omega,
\end{aligned}\right.
\end{equation}
where $\nu$ denotes the outward unit normal to the smooth exterior domain $\Omega$. Specially, they showed that
 after go through the obstacle $K$,
whether an entire solution, behaving like a planar wave front approaching from infinity, can recover its planar front profile uniformly in space depends on the shape of $K$.
Recently, Hoffman et al. \cite{HHV} considered a similar problem for two dimensional lattice differential equations with directionally convex obstacles.

In view of the extensive use of the nonlocal operators to describe the the diffusion phenomenon in biology, physics and chemistry, so much attention has been payed to the study of nonlocal dispersal equations \cite{F2010,PB1997,CC2004,CCR2006,CXF1997,Co2007,CDM2008}. Specially,
the study of nonlocal dispersal equations in exterior domains has attracted much attention recently.
In particular, Cort$\acute{a}$zar et al. \cite{c2012,c20161,c20162} considered the asymptotic behavior for the linear problem. In \cite{BCHV,BC}, Brasseur et al. have established some Liouville results for such nonlocal obstacle problems.
Particularly, they found that the stationary solutions of \eqref{1-1} converging to $1$ as $|x|\rightarrow+\infty$ is indeed $1$ for compact convex obstacle $K$.

When the obstacle $K$ is empty, there have been quite lots of works devoted to the traveling wave solutions and entire solutions for \eqref{1-1} in the recent decade years  \cite{PB1997,CC2004,CP2005,Cha2005,CXF1997,Co2007,CDM2008,KPP,CJ,Wen1982,YH20091
,YH2009,Pan2010,Wang2006,ZLL2009,Guo2005,Hamel1999,Hamel2001,L2010,SZLW2019,ZLW2017,ZLWS2019}.
It follows from (F) and (J) that there exists a solution $\phi$ to the nonlocal wave equation
\begin{equation}\label{1-3}\left\{\begin{aligned}
&\int_{\mathbb{R}}J_1(z-y)[\phi(y)-\phi(z)]dy-c\phi'(z)+f(\phi(z))=0,~z\in\mathbb{R},\\
&\phi(-\infty)=0,~\phi(+\infty)=1,\\
&0<\phi(z)<1,~z\in\mathbb{R}
\end{aligned}\right.
\end{equation}
with $c>0$, where $J_1=\int_{\mathbb{R}^{N-1}}J(x_1,y_1,y_2,...y_N)dy$. In fact, $\phi$ is the traveling wave solution of nonlocal type equation \eqref{1-1} when $K$ is empty, see \cite{PB1997,Sun2011}.
However, if the domain is not the whole space (such as \eqref{1-1}, \eqref{1-2}), there is no classical traveling wave front. Therefore, it is naturally to consider the generalization of traveling fronts.  In fact, such extensions have been introduced in \cite{HM2002,Shen2004,BH2007,BH2012}. Meanwhile, the transition wave front introduced by Berestycki and Hamel in \cite{BH2007,BH2012}, as a fully general notion of traveling front, has been widely established in many works \cite{Bu2018,Hamel20161,Hamel20162,Mellet2009,Nolen2012,Shen2011,Shen20173,Shen20174,Sheng2018,
Zlatos2012,Zlatos2013,Zlatos2017}.
It is interesting to point out that the entire solution constructed in \cite{BHM,HHV} is indeed a generalized transition front.

In the present paper, we are interested to consider the nonlocal dispersal problem \eqref{1-1} in exterior domains.
The main ingredient of this paper is to obtain a unique entire solution of \eqref{1-1} which behaves as planar wave fronts as time to infinity. More precisely, we first prove the existence and uniqueness of the entire solution like a planar wave front at negative infinity time by sub- and super-solutions method.
Moreover, we find  that the entire solution also approaches planar wave fronts as $x$ is far away from  $K$.
 Finally, we shall investigate the procedure how the front goes through $K$ and eventually recovers its shape.
 The investigation of \eqref{1-1} is different to \eqref{1-2} due to the lack of compactness of the nonlocal  operators. Therefore, additional difficulties appear in the study of entire solutions.
  Indeed, the compactness of the solution plays an important role in constructing the uniqueness entire solution and in study of the asymptotic behaviors as $|x|\rightarrow+\infty$. So motivated by the recent work of Li et al. \cite{L2010}, we establish the Lipschitz continuity in space variable $x$ of entire solutions of the nonlocal problem \eqref{1-1} in exterior domains. Then we can discuss the uniqueness and asymptotic behaviors of entire solutions of \eqref{1-1}. Meanwhile, the appearance of convolution term and interior domain leads to the planar wave fronts are not the solutions of \eqref{1-1}, which causes much trouble in verifying the sub- and super-solutions to construct the entire solution. Particularly, the nonlocal dispersal equations admit no explicit fundamental solutions as Laplacian dispersal equation. Therefore, the sub- and super-solutions in \cite{BHM,HHV} are not suitable here to study the asymptotic behaviors of such an entire solution as time goes positive infinity.  Consequently, we have to construct new sub- and super-solutions inspired by \cite{BHM,HHV} to investigate the asymptotic behaviors of such an entire solution as time goes positive infinity for the nonlocal dispersal equation \eqref{1-1}. At last,
our results show that the geometric shape of the interior domain affects the propagation of planar wave fronts.

Now we are ready to state the main result of this paper.

\begin{theorem}\label{th1-1}
Assume that $(F)$ and $(J)$ hold. Let $\phi$ be the unique solution of \eqref{1-3} with $c>0$. If $K$ is convex, then there exists a unique entire solution $u(x,t)$ to \eqref{1-1} satisfying $0<u(x,t)<1$ and $u_t(x,t)>0$ for all $(x,t)\in\overline\Omega\times\mathbb{R}$. Moreover, we have
\[u(x,t)-\phi(x_1+ct)\rightarrow0\]
as $t\rightarrow-\infty$ uniformly in $x\in\overline\Omega$, and as $|x|\rightarrow+\infty$ uniformly in $t\in\mathbb{R}$, where $\phi(x_1+ct)$ is the traveling wave solution satisfying \eqref{1-3}. In addition, if there only exits one zero point of $f$ in $(0,1)$, then
\[u(x,t)-\phi(x_1+ct)\rightarrow0~\text{as}~t\rightarrow+\infty~\text{uniformly in}~x\in\overline{\Omega}.\]
\end{theorem}
\begin{remark}{\rm
It follows from \cite[Theorem 1.6]{BHM} that the solution $u(x,t)$ constructed in Theorem \ref{th1-1} is a general transition solution to \eqref{1-1}. We conjecture that there exists a unique global mean speed $c_t$ of the transition front solution constructed in this paper, which will be remained for further study.
}
\end{remark}

In fact, Theorem \ref{1-1} is also true for some small perturbations of convex interior domains $K$ by \cite{BCHV}, under some additional conditions.
For reader's convenience, we state some definitions before the conclusions as follows.

\begin{definition}
Let $\alpha\in(0,1]$, let $K\subset\mathbb{R}^N$ be a compact convex with nonempty interior ($\partial K$ is then automatically of class $C^{0,\alpha}$) and let $(K_\epsilon)_{0<\epsilon\leq1}\subset\mathbb{R}^N$ be a family of compact, simply connected sets having $C^{0,\alpha}$ boundary. We say that $(K_\epsilon)_{0<\epsilon\leq1}$ is a family of $C^{0,\alpha}$ deformations of $K$ if the following conditions are fulfilled:
\begin{itemize}
\item[(i)] $K\subset K_{\epsilon_1}\subset K_{\epsilon_2}$ for all $0<\epsilon_1\leq\epsilon_2\leq1$;

\item[(ii)] $K_\epsilon\rightarrow K$ as $\epsilon\rightarrow0$ in $C^{0,\alpha}$, in the sense that there exist $r>0$, $p\in\mathbb{N}$, $p$ rotations $(R_i)_{1\leq i\leq p}$ of $\mathbb{R}^N$, $p$ points $(z_i)_{1\leq i\leq p}$ of $\partial K$ and $p$ functions $(\psi_{i})_{1\leq i\leq p}$ and $p$ families of functions $(\psi_{i,\epsilon})_{1\leq i\leq p,~0<\epsilon\leq1}$ of class $C^{0,\alpha}\left(B_r^{N-1}\right)$ describing $\partial K$ and $\partial K_\epsilon$ and such that
\[\|\psi_i-\psi_{i,\epsilon}\|_{C^{0,\alpha}
\left(B_r^{N-1}\right)}~\text{as}~\epsilon\rightarrow0,~
\text{for every}~1\leq i\leq p.\]

\end{itemize}
\end{definition}

The conclusion for the perturbations of compact convex $K$ is as follows.
\begin{theorem}\label{th1-4}
Under the settings of Theorem \ref{th1-1}, let $\alpha\in(0,1]$ and $K\subset\mathbb{R}^N$ be a compact convex set with non-empty interior and let $\{K_\epsilon\}_{0<\epsilon\leq1}$ be a family of $C^{0,\alpha}$ deformations of $K$. Assume that
\[\max\limits_{s\in[0,1]}f'(s)<\inf\limits_{0<\epsilon\leq1}\inf\limits_{x\in\mathbb{R}^N\setminus K_\epsilon}
\|J(x-\cdot)\|_{L^1(\mathbb{R}^N\setminus K_\epsilon)}.\]
Then there exists $\epsilon_0\in(0,1]$ such that Theorem \ref{th1-1} also hold with $K$ replaced by $K_\epsilon$  for $\epsilon\in(0,\epsilon_0]$.
\end{theorem}

In this paper, we establish the existence of entire solutions of the nonlocal dispersal equation \eqref{1-1} with the interior domain $K$. Compare to the Laplacian diffusion problem, the geometry of the set $K$ could be arbitrary provided it is convex. It is naturally to ask if the kernel function is not compacted supported, whether the entire solution we construct in this paper exists. We conjecture that when the obstacle is not convex we say any compact set of $\mathbb{R}^N$, the entire solution of \eqref{1-1} can still recover its shape, but it converges to the nonconstant stationary solution in any bounded subset of $\mathbb{R}^N$ containing $K$.
We shall study these in a future work.

This paper is organized as follows. In Section 2, we consider the Cauchy problem and establish the comparison principle for \eqref{1-1}. The unique entire solution is constructed in Section 3. In Section 4, we study the behaviors of the entire solution far away from $K$ in the space. Section 5 is devoted to discussing the asymptotic behavior of the entire solution as time goes to positive infinity.

\section{Preliminaries}
\subsection{The nonlocal Cauchy problem}
\noindent

We first consider the nonlocal Cauchy problem
\begin{equation}\label{2-1}\left\{\begin{aligned}
&u_t(x,t)=\int_{\Omega}J(x-y)[u(y,t)-u(x,t)]dy+f(u),~x\in\Omega,~t\geq0, \\
&u(x,0)=u_0(x),~x\in\Omega,
\end{aligned}\right.
\end{equation}
where $\Omega$ is a subset of $\mathbb{R}^N$. We call $u(x,t)$ a solution of \eqref{2-1}, if it satisfies
\begin{equation}\label{2-2}
u(x,t)=u_0(x)+\int_0^t\int_{\Omega}J(x-y)[u(y,s)-u(x,s)]dyds+\int_0^tf(u(x,s))ds
\end{equation}
for $x\in\Omega$ and $t\geq0$. Then we have the following theorem.
\begin{theorem}\label{t1}
Suppose that $(J)$ holds and $f\in C^{1,1}(\mathbb{R})$. Then, for any $u_0\in L^1(\Omega)$,  there exists a unique solution $u\in C(L^1(\Omega),[0,t_0])$ to \eqref{2-2} for some $t_0>0$.
\end{theorem}
\begin{proof}
Define the operator $\mathcal{T}$ as follows \[\mathcal{T}w=u_0(x)+\int_0^t\int_{\Omega}J(x-y)[w(y,s)-w(x,s)]dyds+\int_0^tf(w(x,s))ds.\]
We can see that
\[\|\mathcal{T}w\|\leq\|u_0(x)\|+[2+\sup\limits_{\tau\in\mathbb{R}}f'(\tau)]t_0\|w(x,t)\|,\]
 which means $\mathcal{T}$ maps $C(L^1(\Omega),[0,t_0])$ into $C(L^1(\Omega),[0,t_0])$.
On the other hand, we have
\begin{equation*}\begin{split}
&\|\mathcal{T}u(x,t)-\mathcal{T}v(x,t)\|\\
=&\left\|\int_0^t\int_{\Omega}J(x-y)[u(y,s)-v(y,s)+v(x,s)-u(x,s)]dyds
+\int_0^t[f(u(x,s))-f(v(x,s))ds\right\| \\
\leq& 2t_0\|u(x,t)-v(x,t)\|+t_0\sup\limits_{\tau\in\mathbb{R}}f'(\tau)\|u(x,t)-v(x,t)\|\\
\leq&[2+\sup\limits_{\tau\in\mathbb{R}}f'(\tau)]t_0\|u(x,t)-v(x,t)\|.
\end{split}
\end{equation*}
 Let $t_0$ be sufficiently small such that $[2+\sup\limits_{\tau\in\mathbb{R}}f'(\tau)]t_0<1$. Then we obtain that $\mathcal{T}$ is a strict contraction mapping in $C(L^1(\Omega),[0,t_0])$.
\end{proof}

To extend the solution to $[0,+\infty)$ we may take $u(x,t_0)\in L^1(\Omega)$ for the initial datum and obtain a solution in $[t_0,2t_0]$. Iterating this procedure we get a solution in $[0,+\infty)$.

\subsection{Comparison principle}
\noindent

\begin{theorem}\label{t2.1'}
Suppose that the assumptions of Theorem \ref{th1-1} hold. Furthermore, let
  $u(x,t)$, $v(x,t)\in C^1(L^\infty(\Omega,\mathbb{R}),[0,+\infty))$ are uniformly bounded and satisfy
\begin{equation*}\left\{\begin{aligned}
& \frac{\partial u(x,t)}{\partial t}-\left(\int_{\Omega}J(x-y)[u(y,t)-u(x,t)]dy\right)+f(u(x,t))\geq0,~(x,t)\in\Omega\times(0,+\infty),\\
& u(x,0)\geq0,~x\in\Omega,
\end{aligned}\right.
\end{equation*}
\begin{equation*}\left\{\begin{aligned}
& \frac{\partial v(x,t)}{\partial t}-\left(\int_{\Omega}J(x-y)[v(y,t)-v(x,t)]dy\right)+f(v(x,t))\leq0,~(x,t)\in\Omega\times(0,+\infty),\\
& v(x,0)\leq0,~x\in\Omega,
\end{aligned}\right.
\end{equation*}
respectively, where $u(x,0),~v(x,0)\in L^{\infty}(\Omega)$. Then,
 \[u(x,t)\geq v(x,t)~\text{in}~ \Omega\times[0,+\infty).\]
\end{theorem}
\begin{proof}
Define $W(x,t)=u(x,t)-v(x,t)$, we have
\begin{equation}\label{2-3}\begin{split}
W_t(x,t)&\geq\int_{\Omega}J(x-y)[W(y,t)-W(x,t)]dy+f(u)-f(v)\\
&=\int_{\Omega}J(x-y)[W(y,t)-W(x,t)]dy+F(x,t)W(x,t),
\end{split}\end{equation}
where
\[F(x,t)=\int_0^1f'(v(x,t)+\theta W(x,t))d\theta.\]
Suppose that there exist $t_*>0$ and $x_*\in\Omega$ such that $W(x_*,t_*)<0$ and let $\theta_*=-W(x_*,t_*)$. Moreover, Picking $\epsilon>0$ and $K'>0$ such that
$\theta_*=\epsilon e^{2K't_*}$, we then define
\[T_*=\sup\left\{t\geq0\mid W(x,t)>-\epsilon e^{2K't}~\text{for}~x\in\Omega\right\}.
\]
Note that, $W(x,\cdot)\in C^1(0,\infty)$ gives that $0<T_*\leq t_*$. Besides, we have
\[\inf\limits_{\Omega}W(x,T_*)=-\epsilon e^{2K'T_*}.
\]
Without loss of generality, we may assume that $0\in\Omega$ and $W(0,T_*)<-\frac{7}{8}\epsilon e^{2K'T_*}$.

Consider now the function
\begin{equation*}
  W^-(x,t,\beta)=-\epsilon\left(\frac{3}{4}+\beta Z(x)\right)e^{2K't},
\end{equation*}
in which $\beta>0$ is a parameter and $Z\in L^{\infty}(\Omega,\mathbb{R})$ with $Z(0)=1$, $\lim\limits_{|x|\rightarrow+\infty}Z(x)=3,~1\leq Z(x)\leq3$. Take $\beta_*\in(\frac{1}{8},\frac{1}{4}]$ as the minimal value of $\beta$ for which $W(x,t)\geq W^-(x,t)$ holds for all $(x,t)\in\Omega\times[0,T_*]$. Since
\[\lim\limits_{|x|\rightarrow+\infty}W^-(x,t,\beta_*)=-\epsilon\left(\frac{3}{4}+3\beta_*\right)
e^{2K't}<-\frac{9}{8}\epsilon e^{2K't},\]
there exist $x^*\in\Omega$ and $0<t_0<T_*$ such that $W(x^*,t_0)=W^-(x^*,t_0,\beta_*)$. The definition of $\beta_*$ now implies that
\[W_t(x^*,t_0)\leq W_t^-(x^*,t_0,\beta_*).\]
In addition,
\[\int_\Omega J(x^*-y)[W(y,t_0)-W(x^*,t_0)]dy\geq\int_\Omega J(x^*-y)[W^-(y,t_0,\beta_*)-W^-(x^*,t_0,\beta_*)]dy,
\]
by the fact that $W(x,t)\geq W^-(x,t,\beta_*)$ for all $(x,t)\in\Omega\times[0,T_*]$. It follows from \eqref{2-3} that
\begin{equation*}\begin{split}
  -\frac{7}{4}\epsilon e^{2K't_0}\geq W^-_t(x^*,t_0,\beta_*)\geq& W_t(x^*,t_0)\\
  \geq&\int_\Omega J(x^*-y)[W^-(y,t_0,\beta_*)-W^-(x^*,t_0,\beta_*)]dy\\
  &+F(x^*,t_0)W^-(x^*,t_0,\beta_*).
\end{split}\end{equation*}
In particular, we obtain
\begin{equation*}
\begin{split}
-\frac{7}{4}\epsilon K'e^{2K't_0}&\geq\beta_*\epsilon\int_\Omega J(x^*-y)[Z(y)-Z(x^*)]dye^{2K't_0}+F(x^*,t_0)W^-(x^*,t_0,\beta_*)\\
& \geq-\epsilon\left[2\beta_*+\left(3\beta_*+\frac{3}{4}\right)\right]e^{2K't_0}\\
& \geq-2\epsilon e^{2K't_0}.
   \end{split}
\end{equation*}
 This leads to a contradiction upon choosing $K'$ to be sufficiently large.
\end{proof}
\begin{corollary}\label{l1}
Under the assumptions of Theorem \ref{t2.1'}, let $u(x,t)$ and $v(x,t)$ be solutions of \eqref{1-1} with initial values $u(x,0)$ and $v(x,0)$, respectively. If $u(x,0)\geq v(x,0)$ and $u(x,0)\not\equiv v(x,0)$, then $u(x,t)>v(x,t)$ for all $x\in\Omega$ and $t\geq 0$.
\end{corollary}
\begin{proof}
Denote
\[w(x,t)=u(x,t)-v(x,t),\quad\tilde{w}(x,t)=e^{pt}w(x,t)+\epsilon t.\]
Suppose that there exists $(x^0,t^0)\in\overline{\Omega}\times[0,+\infty)$ such that $w(x^0,t^0)=0$ and
let $p>0$ be large enough such that
\[p+F(x,t)\geq0 ~\text{for all}~ (x,t)\in\Omega\times[0,+\infty),\]
 where $F(x,t)$ is defined in the proof of Theorem \ref{t2.1'}.
Then there exists $(x_*,t_*)$ such that
\[\tilde{w}(x_*,t_*)=\min\tilde{w}(x,t)\geq0,\quad\tilde{w}_t(x_*,t_*)=0.\]
 On the other hand, we know that
\begin{align*}
\tilde{w}_t(x_*,t_*)=&pw(x_*,t_*)e^{pt_*}+e^{pt_*}w_t(x_*,t_*)+\epsilon\\
>&pw(x_*,t_*)e^{pt_*}+e^{pt_*}\left(\int_{\Omega}J(x_*-y)[w(y,t_*)-w(x_*,t_*)]dy+F(x_*,t_*)w(x_*,t_*)\right)\\
=&\int_{\Omega}J(x_*-y)[\tilde{w}(y,t_*)-\tilde{w}(x_*,t_*)]dy+(p+F(x_*,t_*))w(x_*,t_*)e^{pt_*}\\
\geq&0.
\end{align*}
Thus we get a contradiction.
\end{proof}

\subsection{Traveling waves}
\noindent

It follows from the results of  Bates et al.  \cite{PB1997} and Sun et al. \cite{Sun2011} that the solution $\phi(z)$ of \eqref{1-3} satisfies
\begin{equation}\label{2-4}\begin{split}
&\alpha_0e^{\lambda z}\leq\phi(z)\leq\beta_0e^{\lambda z},~z\leq0,\\
&\alpha_1e^{-\mu z}\leq1-\phi(z)\leq\beta_1e^{-\mu z},~z>0,
\end{split}\end{equation}
where $\alpha_0,~\alpha_1,~\beta_0$ and $\beta_1$ are some positive constants, $\lambda$ and $\mu$ are the positive roots of
\begin{equation*}
  c\lambda=\int_{\mathbb{R}}J_1(y)e^{-\lambda y}dy-1+f'(0),~c\mu=\int_{\mathbb{R}}J_1(y)e^{-\mu y}dy-1+f'(1).
\end{equation*}
Moreover, we have
\begin{equation*}
\begin{split}
& \gamma_0e^{\lambda z}\leq\phi'(z)\leq\delta_0e^{\lambda z},~z\leq0, \\
& \gamma_1e^{-\mu z}\leq\phi'(z)\leq\delta_1e^{-\mu z},~z>0
\end{split}
\end{equation*}
for some constants $\gamma_0,~\gamma_1,~\delta_0$ and $\delta_1>0$.
At last, note that $f\in C^{1,1}([0,1])$, there exists some $L_f>0$ such that
\begin{equation*}
  |f(u+v)-f(u)-f(v)|\leq L_fuv~\text{for}~0\leq u,v\leq1.
\end{equation*}

\section{Existence and uniqueness of entire solution}
\noindent

This section is devoted to studying the existence and uniqueness of entire solution of \eqref{1-1} that behaves as a planar traveling front until it approaches the interior domain $K$.
 In what follows, we denote $\theta_0$ the largest positive constant such that
\begin{equation*}
f(\tau)\leq0 ~\text{for}~0\leq\tau\leq\theta_0.
\end{equation*}
By the assumption $(F)$, we have $0<\theta_0<1$. We normalize the function $\phi$ in \eqref{1-3} by
$\phi(0)=\theta_0.$ In addition, we further assume that $\phi''(\xi)\geq0$ for $\xi\leq0$. Then we can see the solution $\phi$ of \eqref{1-3} is unique.
The main result of this section is stated as follows.
\begin{theorem}\label{t2.1}
Assume that $(F)$ and $(J)$ hold. If $K\subset\{x\in\mathbb{R}^N:~x_1\leq0\}$, then there exists a unique entire solution $U(x,t)$ of \eqref{1-1} satisfying
 \[0<U(x,t)<1,\quad U_t(x,t)>0~\text{for all}~(x,t)\in\overline{\Omega}\times\mathbb{R}\]
 and
\begin{equation}\label{3-1}
U(x,t)\rightarrow\phi(x_1+ct) ~\text{as}~t\rightarrow-\infty~\text{uniformly in}~x\in\overline{\Omega}.
\end{equation}
\end{theorem}

In this section, the radial symmetry of $J$ can be released to $J(x)=J(-x)$. Moreover, there is no need the convexity and compactness of the obstacle, while the boundedness of $K$ is necessary. We proof Theorem \ref{t2.1} by sub- and super-solutions.

\subsection{Construction of the entire solution}
\noindent

To establish the entire solution,  we shall construct some suitable sub- and super-solutions. Inspired by \cite{BHM},
we define the sub-solution
\begin{equation*}
W^-(x,t)=\left\{\begin{aligned}
&\phi(x_1+ct-\xi(t))-\phi(-x_1+ct-\xi(t)),~x_1\geq0,\\
&0,~x_1<0,
\end{aligned}\right.
\end{equation*}
and the super-solution
\begin{equation*}
W^+(x,t)=\left\{\begin{aligned}
&\phi(x_1+ct+\xi(t))+\phi(-x_1+ct+\xi(t)),~x_1\geq0,\\
&2\phi(ct+\xi(t)),~x_1<0,
\end{aligned}\right.
\end{equation*}
here $\xi(t)$ is the solution of the following equation
\begin{equation*}
\dot{\xi}(t)=Me^{\lambda_0(ct+\xi)}, ~t<-T,~\xi(-\infty)=0,
\end{equation*}
where $M,~\lambda_0$ and $T$ are positive constants to be specified later.
A direct calculation yields that
 \[\xi(t)=\frac{1}{\lambda_0}\ln\frac{1}{1-c^{-1}Me^{\lambda_0 ct}}.\]
For the function $\xi(t)$ to be defined, one must have $1-c^{-1}Me^{\lambda_0 ct}>0$. Besides, we suppose  that
\[ct+\xi(t)\leq0 ~\text{for}~-\infty<t\leq T.
\]
Thus set $T:=\frac{1}{\lambda_0 c}\ln\frac{c}{c+M}$.
Then the following proposition holds.
\begin{proposition}\label{t2.2}
Assume that $\lambda_0<\min\{\lambda,k_\phi\}$ with given $k_\phi>0$ and $K\subset\mathbb{R}^N\setminus$supp$(J)$.
Then there exists sufficiently large $M>0$ such that $W^-$ and $W^+$ are sub- and super-solutions of \eqref{1-1} in the time range $-\infty<t\leq T_1$ for some $T_1\in(-\infty, T]$.
\end{proposition}
The proof of this lemma will be given in Appendix for the coherence of this paper.

Now we are in position to construct the entire solution. Let $u_n(x,t)$ be the unique solution of \eqref{1-1} for $t\geq-n$ with initial data
\[u_n(x,-n)=W^-(x,-n).\]
Since $W^-(x,t)$ is a sub-solution, it is not difficult show that $\{u_n(x,t)\}_{n=1}^\infty$ is a nondecreasing sequence with respect to $t$. Choose some constant $T^*>0$ such that $c>\dot{\xi}(t)$ for $t\leq-T^*$. In the following discussing, without any loss of generality, we assume that $n\geq T^*$. Then, we have
\begin{equation*}
\frac{\partial u_n(x,t)}{\partial t}=\int_\Omega J(x-y)[u_n(y,t)-u_n(x,t)]dy+f(u_n(x,t))
\end{equation*}
for $t\geq-n$ and $x\in\Omega$.
Since $\frac{\partial W^-(x,t)}{\partial t}=0$ for $x_1\leq0$ and
\[W^-_t(x,t)=(c-\dot{\xi}(t))(\phi'(x_1+ct-\xi(t))-\phi'(-x_1+ct-\xi(t)))\geq0\]
for $0<x_1\leq|ct-\xi(t)|$ and $t<-T^*$ duo to $\phi''(\xi)\geq0$ for $\xi\leq0$, it follows that
\begin{equation*}
\begin{split}
\frac{\partial u_n(x,-n)}{\partial t}&=\int_\Omega J(x-y)[u_n(y,-n)-u_n(x,-n)]dy+f(u_n(x,-n)) \\
    &\geq \frac{\partial W^-(x,-n)}{\partial t}\\
    &\geq0
\end{split}
\end{equation*}
for all $x_1\leq|cn+\xi(-n)|$.
Furthermore, by the comparison principle, $u_n(x,t)$ satisfies
\[\frac{\partial u_n(x,t)}{\partial t}>0~\text{for all}~x_1\leq|cn+\xi(-n)|,~0<u_n(x,t)<1~\text{for all}~t\geq-n,~x\in\Omega,\]
and
\[W^-(x,t)<u_n(x,t)<W^+(x,t)~\text{for all}~-n< t\leq T^*~\text{and}~x\in\Omega.\]
For each fixed $x\in\Omega$, since $\left\{u_n(x,t),\frac{\partial u_n(x,t)}{\partial t}\right\}_{n=1}^{+\infty}$ is well-defined for large $n$ and equicontinuous with $t$, there exists a subsequence, still denoted by $\left\{u_n(x,t),\frac{\partial u_n(x,t)}{\partial t}\right\}_{n=1}^{+\infty}$, such that
\begin{equation}\label{3-2}
\left(u_n(x,t),\frac{\partial u_n(x,t)}{\partial t}\right)\rightarrow(u(x,t),u_t(x,t))~\text{as}~n\rightarrow+\infty,
\end{equation}
where the convergence is locally uniform in $t\in\mathbb{R}$. Moreover, via diagonalization, define  $U:~\Omega\times\mathbb{R}\rightarrow L^\infty(\Omega;\mathbb{R})$ as the limit of
a subsequence for which \eqref{3-2} holds. Then we have
\[U_t(x,t)=\int_\Omega J(x-y)(U(y,t)-U(x,t))dy+f(U(x,t)),\]
and
\begin{equation*}
U_t(x,t)\geq0,~0\leq U(x,t)\leq1.
\end{equation*}
Besides, it follows from the definition of $W^-(x,t)$ that
\[\sup\limits_{x\in\Omega}|U(x,t)-\phi(x_1+ct)|\rightarrow0~\text{as}~t\rightarrow-\infty.\]
In particular, note that $U(x,t)$ is not a constant, by Corollary \ref{l1} there hold
\[U_t(x,t)>0~\text{and}~0<U(x,t)<1.\]
Furthermore, inspired by \cite{L2010}, we can show that $U(x,t)$ satisfies the following proposition.
\begin{proposition}\label{l2.4}
Let $U(x,t)$ be the entire solution in Theorem \ref{t2.1}. Then $U(x,t)$ satisfies
\begin{equation}\label{3-3}
|U(x+\eta,t)-U(x,t)|\leq M'\eta,
\end{equation}
and
\begin{equation}\label{3-4}
\left|\frac{\partial U(x+\eta,t)}{\partial t}-\frac{\partial U(x,t)}{\partial t}\right|\leq M''\eta
\end{equation}
for $M',~M''>0$.
\end{proposition}
\begin{proof}
Since $\int_{\mathbb{R}^N}J(x)dx=1,~J(x)\geq0$ and $J(x)$ is compactly supported, we have $J'\in L^1(\mathbb{R}^N)$. Furthermore, we get
\begin{equation*}
\int_{\Omega}|J(x+\eta)-J(x)|dx=\int_{\Omega}\int_0^1|\nabla J(x+\theta\eta)\eta|d\theta dx
\leq L_1|\eta|~\text{for some constant}~L_1>0.
\end{equation*}
 Let
 \[m=\inf\limits_{u\in[0,1]}\left(\inf\limits_{x\in\Omega}\int_\Omega J(x-y)dy-f'(u)\right)>0\]
  and $v(t)$ be a solution of the equation
\begin{align*}\left\{\begin{aligned}
&v'(t)=L_1|\eta|-mv(t)~\text{for any}~t>-n,\\
&v(-n)=M|\eta|
\end{aligned}
\right.
\end{align*}
for some $M\geq2\sup\limits_{\xi\in\mathbb{R}}|\phi'(\xi)|$.
In addition, denote $V(x,t)=u_n(x+\eta,t)-u_n(x,t)$, where $u_n(x,t)$ is the solution of \eqref{1-1} with initial value $u_n(x,-n)=W^-(x,-n)$. Then
\[V_t(x,t)\leq\int_\Omega[J(x+\eta-y)-J(x-y)][u_n(y,t)-u_n(x+\eta,t)]dy-\inf\limits_{x\in\Omega}\int_\Omega J(x-y)dyV(x,t)
+f'(\overline{V})V,\]
where $\overline{V}$ is between $u_n(x,t)$ and $u_n(x+\eta,t)$. Consequently, $V(x,t)$ satisfies
\[V_t(x,t)\leq L_1|\eta|-mV(t) ~\text{for}~t>-n ~\text{and}~V(x,-n)\leq M|\eta|.\]
Moreover, $|V(x,t)|\leq v(t)\leq M^*|\eta|$ for any $x\in\Omega,~t\geq-n$ and $M^*=M+\frac{L_1}{m}$.
Indeed,
\[0<v(t)=e^{-m(t+n)}M|\eta|+\frac{L_1\eta}{m}\left(1-e^{-m(t+n)}\right)
<\left(M+\frac{L_1}{m}\right)|\eta|<M^*|\eta|\]
 for any $x\in\mathbb{R}^N,~t\geq-n$. In particular, in view of $f'(s)<1$ for $s\in[0,1]$, there holds
\begin{equation*}\begin{split}
&\left|\frac{\partial u_n(x+\eta,t)}{\partial t}-\frac{\partial u_n(x,t)}{\partial t}\right|\\
\leq&\bigg|\int_{\Omega}[J(x+\eta-y)-J(x-y)]u_n(y,t)dy\bigg|
+\big|[u_n(x+\eta,t)-u_n(x,t)]\big|\\
&+\big|f'(\overline{V})[u_n(x+\eta,t)-u_n(x,t)]\big|\\
\leq&\int_\Omega\big|J(x+\eta-y)-J(x-y)\big|dy+(1+\max\limits_{s\in[0,1]}f'(s))\big|u_n(x+\eta,t)-u_n(x,t)\big|\\
\leq&[L_1+(1+\max\limits_{s\in[0,1]}f'(s))M^*]|\eta|.\\
\end{split}
\end{equation*}
At last, since $u_n(x,t)\rightarrow U(x,t)$ locally uniformly in $t\in\mathbb{R}$ as $n\rightarrow+\infty$, we have
\begin{align*}
|U(x+\eta,t)-U(x,t)|\leq&|U(x+\eta)-u_n(x+\eta,t)|+|u_n(x+\eta)-u_n(x,t)|+|u_n(x,t)-U(x,t)|\\
\leq&(M^*+2)|\eta|.
\end{align*}
Now take $M'=M^*+2$, we can prove the \eqref{3-3} and \eqref{3-4} by taking $M''=L_1+2+(1+\max\limits_{s\in[0,1]}f'(s))M^*$.
\end{proof}

\subsection{Uniqueness of the entire solution}
\noindent

The following result plays an important role in proving the uniqueness of the entire solution.
\begin{lemma}\label{l2.1}
Assume that the settings of Theorem \ref{th1-1} hold. Then for any $\varphi\in(0,\frac{1}{2}]$, there exist constants
$T_\varphi=T_\varphi(\varphi)>1$ and $K_\varphi=K_\varphi(\varphi)>0$ such that
 \[U_t(x,t)\geq K_\varphi~\text{for any}~t\leq-T_\varphi~\text{and}~x\in\Omega_{\varphi}(t),\]
 where
\[\Omega_{\varphi}(t)=\left\{x\in\Omega:~\varphi\leq U(x,t)\leq1-\varphi\right\}.\]
\end{lemma}
\begin{proof}
It is easy to choose $T_\varphi$ and $M_\varphi$ such that
$\Omega_{\varphi}(t)\subset\left\{x\in\Omega:~|x_1+ct|\leq M_\varphi\right\}\subset\{x\in\mathbb{R}^N:~x_1\geq1\}$.
Now suppose there exist sequences $t_k\in(-\infty,T_\varphi]$ and
  $x^k:=(x^k_1,x^k_2,...x^k_N)\in\Omega_{\varphi}(t)$ such that
\[U_t(t_k,x^k)\rightarrow0~\text{as}~k\rightarrow+\infty.\]
Here only two cases happen,
$t_k\rightarrow-\infty$ or $t_k\rightarrow t_*$ for some $t_*\in(-\infty,T_\varphi]$ as $k\rightarrow+\infty$.

For the former case, denote
\[U_k(x,t)=U(x+x^k,t+t_k).\]
By Lemma \ref{l2.4}, $\{U_k(x,t)\}_{k=1}^\infty$ is equicontinuous in $x\in\Omega$ and $t\in\mathbb{R}$. Furthermore,  there exists a  subsequence still denoted by $\{U_k(x,t)\}_{k=1}^\infty$ such that
 \[U_k\rightarrow U_*~\text{as}~k\rightarrow+\infty\]
 and $U_*$ satisfies $\frac{\partial U_*(0,0)}{\partial t}=0$. We further have
 \[\frac{\partial U_*(x,t)}{\partial t}\equiv0~\text{for}~t\leq0.\]
 However, this is impossible because
\[U_*(x,t)=\phi(x_1+ct+a)~\text{for some}~a\in[-M_\eta,M_\eta].\]

For the second case, $x^k_1$ remains bounded by the definition of
 $\Omega_{\varphi}(t)$.  Therefore, we assume that $x^k_1\rightarrow x^*_1$ as
  $k\rightarrow+\infty$ and let
\[U_k(x,t):=U(x+x^k,t).\]
Then, each $U_k(x,t)$ is defined for all
$(x,t)\in(-\infty,T_\varphi]\times\{x\in\mathbb{R}^N\mid x_1\geq-1\}$ by the definition of $\Omega_\varphi(t)$. Similarly, there exists a
 subsequence, again denoted by $\{U_k\}^{\infty}_{k=1}$, such that
 \[U_k\rightarrow U^*~\text{as}~k\rightarrow+\infty~\text{with}~x_1\geq-1\]
  for some function $U^*$ satisfies \eqref{1-1} on $\{x\in\mathbb{R}^N\mid x_1\geq-1\}\times(-\infty,T_\varphi]$. Note that $U_t(x,t)>0$, we have
\[\frac{\partial U^*}{\partial t}(0,t_*)=0,~\frac{\partial U^*}{\partial t}(x,t)\geq0
~\text{for}~(x,t)\in\{x\in\mathbb{R}^N\mid x_1\geq-1\}\times(-\infty,T_\varphi].\]
Then we obtain $\frac{\partial U^*}{\partial t}(x,t)\equiv0$ for $t\leq t_*$, but this is impossible since
\[U^*(x,t)-\phi(x_1+x^*_1+ct)\rightarrow0~\text{as}~t\rightarrow-\infty,~\text{uniformly in}
~{\{x\in\mathbb{R}^N\mid x_1\geq-1\}}.\]
This ends the proof.
\end{proof}
Now we are ready to show the uniqueness of the entire solution. Suppose that there exists another entire solution $V$ of \eqref{1-1} satisfying \eqref{3-1}.
Extend the function $f$ as
 \[f(s)=f'(0)s~\text{for}~s\leq0,\quad f(s)=f'(1)(s-1)~\text{for}~s\geq1\]
  and choose $\eta>0$ sufficiently small such that
\[f'(s)\leq-\omega~\text{for}~s\in[-2\eta,2\eta]\cup[1-2\eta,1+2\eta]~\text{and}~\omega>0.\]
Then for any $\epsilon\in(0,\eta)$ we can find $t_0\in\mathbb{R}$ such that
\begin{equation}\label{3-5}
\|V(\cdot,t)-U(\cdot,t)\|_{L^{\infty}(\Omega)}<\epsilon~\text{for}~-\infty<t\leq t_0.
\end{equation}
For each $t_0\in(-\infty,T_\varphi-\sigma\epsilon]$, define
\[\tilde{U}^+(x,t):=U(x,t_0+t+\sigma\epsilon(1-e^{-\omega t}))+\epsilon e^{-\omega t},\]
and
\[\tilde{U}^-(x,t):=U(x,t_0+t-\sigma\epsilon(1-e^{-\omega t}))-\epsilon e^{-\omega t},\]
where the constant $\sigma>0$ is specified later. Then by \eqref{3-5},
\begin{equation}\label{3-6}
\tilde{U}^-(x,0)\leq V(x,t_0)\leq \tilde{U}^+(x,0)~\text{ for}~ x\in\Omega.
\end{equation}
Next we show that $\tilde{U}^-,\tilde{U}^+$ are sub- and super-solutions in the time range
$t\in[0,T_\varphi-t_0-\sigma\epsilon].$
A straightforward calculation implies that
\begin{equation*}
\begin{split}
 \mathcal{L}\tilde{U}^+=&\sigma\epsilon\omega e^{-\omega t}U_t-\epsilon\omega e^{-\omega t}+f(U)-f(U+\epsilon e^{-\omega t})  \\
 =&\epsilon e^{-\omega t}(\sigma\omega U_t-\omega-f'(U+\theta\epsilon e^{-\omega t})),
\end{split}
\end{equation*}
where $0<\theta<1$.
For any $x\not\in\Omega_\eta(x,t_0+t+\sigma\epsilon(1-e^{-\omega t}))$, we can see
\[U+\theta\epsilon e^{-\omega t}\in[0,2\eta]\cup[1-\eta,1+\eta].\]
Consequently, $f'(U+\theta\epsilon e^{-\omega t})\leq-\omega$, which implies that
\[\mathcal{L}\tilde{U}^+\geq\epsilon e^{-\omega t}(-\omega+\omega)=0.\]
For $x\in\Omega_\eta(x,t_0+t+\sigma\epsilon(1-e^{-\omega t}))$, by Lemma \ref{l2.1}, there holds
\[\mathcal{L}\tilde{U}^+\geq\epsilon e^{-\omega t}(\sigma\omega K_\varphi-\omega-\max\limits_{0\leq s\leq1}f'(s)).\]
As a consequence, $\mathcal{L}\tilde{U}^+\geq0$ provided that $\sigma$ is sufficiently large.

Similarly, we can show $\mathcal{L}\tilde{U}^-\leq0$ in $\Omega\times[0,T_\varphi-t_0-\sigma\epsilon]$. In view of
 this and \eqref{3-6}, we see that
\[\tilde{U}^-(x,t)\leq V(x,t+t_0)\leq \tilde{U}^+(x,t)~\text{for all}~(x,t)\in\Omega\times[0,T_\varphi-t_0
-\sigma\epsilon].\]
Letting $t+t_0$ be replaced by $t$, the inequality above can be rewritten as
\[U(x,t-\sigma\epsilon(1-e^{-\omega(t-t_0)}))-\epsilon e^{-\omega(t-t_0)}\leq V(x,t)\leq U(x,t+\sigma\epsilon
(1-e^{-\omega(t-t_0)}))+\epsilon e^{-\omega(t-t_0)}\]
for all $(x,t)\in\Omega\times[t_0,T_\varphi-\sigma\epsilon]$ and $t_0\in(-\infty,T_\varphi-\sigma\epsilon]$.
 As $t_0\rightarrow-\infty$, we obtain that
\[U(x,t-\sigma\epsilon)\leq V(x,t)\leq U(x,t+\sigma\epsilon) ~\text{for all}~(x,t)\in\Omega\times
(-\infty,T_\varphi-\sigma\epsilon].\]
By the comparison principle, the inequality holds for $t\in\mathbb{R},~x\in\Omega$. Letting $\epsilon\rightarrow0$, we have
$V(x,t)\equiv U(x,t)$.

\section{Behavior far away from the interior domain}
\noindent

In this section, we are going to figure out what the entire solution, constructed in previous section, is like
 far away from the interior domain.
\begin{theorem}\label{t5.2}
Assume that $(F)$ and $(J)$ hold. Let $\phi$ be the unique solution of \eqref{1-3} with $c>0$ and $u(x,t)$ be a solution of
\begin{equation}\label{4-1}\left\{\begin{aligned}
&u_t(x,t)=\int_\Omega J(x-y)[u(y,t)-u(x,t)]dy+f(u),\quad x\in\Omega,~t\in\mathbb{R},\\
&0\leq u(x,t)\leq1,\quad x\in\Omega,~t\in\mathbb{R},
\end{aligned}\right.
\end{equation}
such that
\[\sup\limits_{x\in\bar{\Omega}}|u(x,t)-\phi(x_1+ct)|
\rightarrow0~\text{as}~t\rightarrow-\infty.\]
Then, for any sequence $(x'_n)_{n\in\mathbb{N}}\in\mathbb{R}^{N-1}$ such that $|x'_n|\rightarrow+\infty$ as
 $n\rightarrow+\infty$, there holds
\[u(x_1,x'+x'_n,t)\rightarrow\phi(x_1+ct)~\text{as}~n\rightarrow+\infty~\text{for}~t\in\mathbb{R},\]
locally uniformly with respect to $(x,t)=(x_1,x',t)\in\mathbb{R}^N\times\mathbb{R}$.
\end{theorem}
\begin{proof}
In order to prove this lemma we need to extend the function $f$ as
\[f(s)=f'(0)s~\text{for}~s\leq0,~f(s)=f'(1)(s-1)~\text{for}~s\geq1,\]
 and let $\eta>0$ be sufficiently small such that
 \[f'(s)\leq-\omega~\text{for}~s\in(-\infty,\eta]\cup
 [1-\eta,+\infty),~~\omega=\min\left(\frac{|f'(0)|}{2},\frac{|f'(1)|}{2}\right)>0.\]
  In addition, let $A>0$ be sufficiently large such that
   \[\phi(\xi)\leq\frac{\eta}{2}~\text{for}~\xi\leq-A,~~
  \phi(\xi)\geq 1-\frac{\eta}{2}~\text{for}~\xi\geq A.\]
   Denote $\delta=\min\limits_{[-A,A]}\phi'(\xi)>0$
 and take $T\in\mathbb{R}>0$ such that
 for any $\epsilon\in(0,\frac{\eta}{2})$, we have that
\[|u(x,t)-\phi(x_1+ct)|\leq\epsilon~\text{for all}~t\leq -T~\text{and}~x\in\Omega.\]

Now we are in position to show this lemma. Under the assumptions of Theorem \ref{t5.2}, let $(x'_n)_{n\in\mathbb{N}}\in\mathbb{R}^{N-1}$ be
 a sequence such that $|x'_n|\rightarrow+\infty$ as $n\rightarrow+\infty$. Meanwhile, denote
 $u_n(x,t)=u(x_1,x'+x'_n,t)$ for each $t\in\mathbb{R}$ and $x=(x_1,x')\in\Omega-(0,x'_n)$. Since
  $0\leq u\leq1$, $K$ is compact and $\left\{u_n(x,t),\frac{\partial u_n(x,t)}{\partial t}\right\}_{n=1}^{\infty}$ is equal-continuous by Proposition \ref{l2.4},
  then there exists a subsequence, still denoted by $\left\{u_n(x,t),\frac{\partial u_n(x,t)}{\partial t}\right\}_{n=1}^\infty$, such that
\[u_n(x,t)\rightarrow u(x,t),~\frac{\partial u_n(x,t)}{\partial t}\rightarrow u_t(x,t)~\text{as}~n\rightarrow+\infty\]
 locally uniformly in $\mathbb{R}^N\times\mathbb{R}$. Moreover, we get that
\begin{equation}\label{4-2}u_t=\int_{\mathbb{R}^N}J(x-y)[u(y,t)-u(x,t)]dy+f(u),~0\leq u(x,t)\leq1
~\text{for all}~(x,t)\in\mathbb{R}^N\times\mathbb{R},\end{equation}
since $J$ has compact support, $K$ is compact and $|x'_n|\rightarrow+\infty$ as $n\rightarrow\infty$. In addition, recall that
\[u_n(x,t)-\phi(x_1+ct)\rightarrow0~\text{as}~t\rightarrow-\infty\]
uniformly in $\overline{\Omega}$, the function $u(x,t)$ satisfies
\[|u(x,t)-\phi(x_1+ct)|\rightarrow0~\text{as}~t\rightarrow-\infty,~\text{locally uniformly in}~\mathbb{R}^N.\]

Now define two functions $\underline{u}(x,t)$ and $\overline{u}(x,t)$ as follows
\[\underline{u}(x,t)=\phi(\xi_-(x,t))-\epsilon e^{-\omega(t-t_0)},\overline{u}(x,t)=\phi(\xi_+(x,t))+\epsilon e^{-\omega(t-t_0)},~t\geq t_0,~x\in\mathbb{R}^N,\]
where
\[t_0\leq-T,~\xi_\pm(x,t)=x_1+ct\pm2\epsilon\|f'\|\delta^{-1}\omega^{-1}
\left[1-e^{-\omega(t-t_0)}\right].\]
Then the following lemma holds, whose proof is left to Appendix as a regular argument.
\begin{lemma}\label{l4.2}
The functions $\underline{u}(x,t)$ and $\overline{u}(x,t)$ are sub- and super-solutions to \eqref{4-2} for $t\geq t_0$, respectively.
\end{lemma}
By the comparison theorem and let $\epsilon\rightarrow0$, we have $u(x,t)\equiv\phi(x_1+ct)$. Since the limit is uniquely determined, the whole sequence $\{u_n(x,t)\}_{n\in\mathbb{N}}$ converges to $\phi(x_1+ct)$ locally uniformly in $(x,t)\in\mathbb{R}^N\times\mathbb{R}$ as $n\rightarrow+\infty$. Then Theorem \ref{t5.2} is completed.
\end{proof}
\begin{theorem}\label{t5.3}
Suppose that all the assumptions in Theorem \ref{t5.2} hold. Then the solution $u(x,t)$ of \eqref{4-1} satisfies
\[|u(x,t)-\phi(x_1+ct)|\rightarrow0~\text{as}~|x|\rightarrow+\infty,\]
 locally uniformly for $t\in\mathbb{R}$.
\end{theorem}
\begin{proof}
Extending $f$ as in Theorem \ref{t5.2}. Then
define $f_\delta(u)=f(u-\delta)$, $(c_\delta,\phi_\delta)$ satisfies
\[c_\delta(\phi_\delta)'=\int_{\mathbb{R}^N}J(x-y)
[\phi_\delta(y)-\phi_\delta(x)]dy+f_\delta(\phi_\delta),\]
and
\[\phi_\delta(-\infty)=\delta,~\phi_\delta(+\infty)=1+\delta,\]
 where $c_\delta>0$ and $\delta>0$ is sufficiently small. Now we are going to show the theorem in three steps.

 Step 1. For $x_1\gg1$, since
\[\sup\limits_{x\in\bar{\Omega}}|u(x,t)-\phi(x_1+ct)|
\rightarrow0~\text{as}~t\rightarrow-\infty,\]
there exists some $T^*_1\geq0$ sufficiently large such that
\[|u(x,t)-\phi(x_1+ct)|\leq\frac{\epsilon}{2}~\text{for}~x\in\Omega~\text{and}~t\leq-T^*_1.\]
In particular, for any $x\in\Omega$, let $N_1\gg1$ such that
\[\phi(x_1-cT^*_1)\geq1-\frac{\epsilon}{2},~u(x,-T^*_1)\geq1-\epsilon ~\text{for}~x_1\geq N_1.\]
Since $\phi'>0,~u_t>0$, we have
\[|u(x,t)-\phi(x_1+ct)|<\epsilon~\text{for}~x_1\geq N_1~\text{and}~t\geq -T^*_1.\]

Step 2. For $x_1\ll-1$, let $\delta=\frac{1}{2}\epsilon$. Similar as first step, choose $T_2>0$ and $N_2\gg1$ such that
\[u(x,t)\leq \delta~\text{for}~x_1\leq-N_2~\text{and}~t\leq-T_2,~\text{particularly},~u(x,-T_2)\leq\delta.\]
Apparently, there exists $x_0\in\mathbb{R}$ such that $\phi_\delta(x_0)=1$. Then
\[u(x,-T_2)\leq\phi_\delta(x_1+x_0+N_2)~\text{for}~x\in\Omega.\]
Combining that $\phi_\delta(x,t)$ is increasing, we have
\[\phi_\delta(x_1+c_\delta(t+T_2)+x_0+N_2)
\geq\phi_\delta(-N_2+x_0+N_2)=1~\text{for}~t\geq-T_2~\text{and}~x_1\geq -N_2.\]
Applying the comparison principle on the region $t\geq-T_2$ and $x_1\leq-N_2$ implies that
\[u(x,t)\leq\phi_\delta(x_1+x_0+N_2+c_\delta(t+T_2))~\text{and}~x_1\leq-N_2
~\text{for}~t\geq-T_2.\]
In particular, since $\phi_\delta(-\infty)=\delta=\frac{1}{2}\epsilon$, there exists a $N_3\gg1$ such that for all $\tau\geq0$,
\[0\leq u(x,t)\leq2\delta<\epsilon,~0<\phi(x_1+ct)\leq\epsilon~\text{for all}~x_1\leq-N_3~\text{and}~\tau\geq t\geq-T_2,\]
which again shows
\[|u(x,t)-\phi(x_1+ct)|\leq\epsilon~\text{for}~x_1\leq-N_3,~t\leq\tau.\]

Step 3. For $|x'|\gg1$. From the Theorem \ref{t5.2}, choose $N_4$ such that
\[|u(x,t)-\phi(x_1+ct)|\leq\epsilon~\text{for}~t\in\mathbb{R}\]
holds whenever $|x'|>N_4$ and $x_1\in[-N_3,N_1]$. This finishes the proof.
\end{proof}

\section{The behavior for the large time}
\noindent

In this section we intend to investigate the behavior of the solution constructed in the previous as the time is large.
For this goal, we establish the following result.
\begin{theorem}\label{t6.1}
Assume that $(F)$ and $(J)$ hold. Let $\phi$ be the unique solution of \eqref{1-3} with $c>0$,
$t_0\in\mathbb{R}$ and $u(x,t)$ be a solution of
\begin{equation}\label{5-1}
\left\{
\begin{aligned}
  u_t(x,t)=\int_\Omega J(x-y)[u(y,t)-u(x,t)]dy+f(u),~& x\in\Omega,~t\in[t_0,+\infty),\\
  0\leq u(x,t)\leq1,~\qquad \qquad\qquad\qquad\qquad\qquad\qquad &  x\in\Omega,~t\in[t_0,+\infty).
\end{aligned}
\right.
\end{equation}
Assume that, for any $\epsilon>0$, there exists a $t_\epsilon\geq t_0$ and a compact set $K_\epsilon\subset\overline{\Omega}$ such that
\[|u(x,t_\epsilon)-\phi(x_1+ct_\epsilon)|\leq\epsilon~\text{for all}~x\in\overline{\Omega\backslash K_\epsilon}\]
and
\[u(x,t)\geq1-\epsilon~\text{for all}~t\geq t_\epsilon~\text{and}~x\in\partial\Omega=\partial K.\]
Then
\[\sup\limits_{x\in\overline{\Omega}}|u(x,t)-\phi(x_1+ct)|\rightarrow0~\text{as}~t\rightarrow+\infty.\]
\end{theorem}
The most important ingredient of the proof is to construct suitable sub- and super-solutions. This progress is such cumbersome that  will be divided several parts.
\subsection{Sub-solution}
\noindent

In this part, we construct a sub-solution to \eqref{5-1}. Now define
\[\tilde{u}(x,t)=u(x,t-1+t_\epsilon), ~u^-(x,t)=\phi(x_1+c(t-1+t_\epsilon)-\theta(x',t)-Z(t))-z(t),\]
 where $\theta(x',t)=\beta t^{-\alpha}e^{\frac{-|x'|}{\gamma t}}$, $Z(t)=K_z\int_0^tz(\tau)d\tau$ and $\gamma>1$ is some constant, $K_z>0$ is large enough. Meanwhile, $z(t)$ will be defined in Theorem \ref{l2.2} later.
It follows from the definition of $u^-(x,t)$ that
\begin{align*}
u^-(x,1)\leq\phi(x_1+ct_\epsilon-\beta e^{\frac{-|x'|}{\gamma}}
-Z(1))\leq\tilde{u}(x,1)=u(x,t_\epsilon)
\end{align*}
for $x\in K_\epsilon.$ Thanks to $\phi(-\infty)=0,~\phi'>0$, $0<u(x,t)<1$ and $\min\limits_{x\in K_\epsilon}u(x,t)>0$, the last inequality holds provided that $\beta$ is sufficiently large. If $x\in\mathbb{R}^N\backslash K_\epsilon$, then
\[u^-(x,1)\leq\phi(x_1+ct_\epsilon)-z(1)\leq\phi(x_1+ct_\epsilon)-\epsilon\leq u(x,t_\epsilon)=\tilde{u}(x,1),\]
by $z(1)\geq\frac{1}{2}\epsilon_1=\epsilon$ (from Remark \ref{r1}).
As a consequence, $u^-(x,1)\leq\tilde{u}(x,1)$ for any $x\in\overline{\Omega}$.
\begin{lemma}\label{l5.1}
The inequality $\mathcal{L}u^-(x,t)\leq0$ holds for $x\in\Omega$ and $t\geq1$, where
\[\mathcal{L}u^-(x,t)=u^-_t(x,t)-\int_\Omega J(x-y)[u^-(y,t)-u^-(x,t)]dy-f(u^-).\]
\end{lemma}
\begin{proof}
Since $u^-(x,t)=\phi(\xi(x,t))-z(t)$, where $\xi(x,t)=x_1+c(t-1+t_\epsilon)
-\beta t^{-\alpha}e^{\frac{-|x'|^2}{t\gamma}}-Z(t)$, we have
\[u^-_t(x,t)=\phi'(\xi(x,t))(c-\theta_t(x',t)-Z')-z'(t),\]
and
\[\int_\Omega J(x-y)[u^-(y,t)-u^-(x,t)]dy=\int_\Omega J(x-y)[\phi(\xi(y,t))-\phi(\xi(x,t))]dy.\]
Denote $\mathcal{D}\phi=\int_{\mathbb{R}^N}J(x-y)[\phi(\xi(y,t))-\phi(\xi(x,t))]dy$. Then, applying mean value theorem, we get that
\begin{align*}
\mathcal{D}\phi
=&\int_{\mathbb{R}^N}J(y)[\phi(\xi(x,t)-y_1)-\phi(\xi(x,t))]dy+\int_{\mathbb{R}^N}J(y)
[\phi(x_1-y_1+c(t-1+t_\epsilon)\\
&-\beta t^{-\alpha}e^{\frac{-|x'-y'|^2}{t\gamma}})-\phi(\xi(x,t)-y_1)]dy\\
\geq&c\phi'(\xi(x,t))-f(\phi(\xi(x,t)))
-C^0\phi'(\xi(x,t))\beta t^{-\alpha}
\int_{\mathbb{R}^N}J(y)|y'|\frac{2|x'-\tilde{\theta}y'|}{t\gamma}
e^{-\frac{|x'-\tilde{\theta}y'|^2}{t\gamma}}dy,
\end{align*}
where $0<\tilde{\theta}<1$. In particular,
\[\mathcal{D}\phi\geq c\phi'(\xi(x,t))-f(\phi(\xi(x,t)))-C't^{-\alpha-1}\phi'(\xi(x,t)).\]
Therefore, we have
\begin{equation*}
\begin{split}
\mathcal{L}u^-(x,t)=&u^-_t(x,t)-\int_\Omega J(x-y)[u^-(y,t)-u^-(x,t)]dy-f(u^-)\\
\leq&f(\phi(\xi(x,t)))-f(\phi(\xi(x,t))-z(t))+(C't^{-\alpha-1}-Z'(t)-\theta_t(x',t))\phi'(\xi(x,t))\\
&+\int_KJ(x-y)[\phi(\xi(y,t))-\phi(\xi(x,t))]dy-z'(t).
\end{split}
\end{equation*}
Similarly as previous, one can get
\[\int_KJ(x-y)[\phi(\xi(y,t)-\phi(\xi(x,t))]\geq-C^Kt^{-\alpha-1}\phi'(\xi(x,t))~\text{for some}~C^K>0.\]
It follows that
\begin{align*}
\mathcal{L}u^-(x,t)\leq& f(\phi(\xi(x,t)))-f(\phi(\xi(x,t))-z(t))\\
&+\left[(C'+C^K)t^{-\alpha-1}-Z'(t)-\theta_t(x',t)
\right]\phi'(\xi(x,t))-z'(t).
\end{align*}
Now we go further to show $\mathcal{L}u^-(x,t)\leq0$ in two cases.

Case 1. We assume that $|\xi(x,t)|\gg1$ such that $\phi(\xi(x,t))\in[0,\eta]\cup[1-\eta,1]$, where $\eta$ is sufficiently small  to ensure that $f'(s)\leq-\sigma<0$ for any $s\in[0,\eta]\cup[1-\eta,1]$ with $\sigma>2\eta_z$. Then, since the function $z(t)$ constructed in Theorem \ref{l2.2} satisfies  \[z'(t)\geq-\eta_zz(t),\quad z(t)\geq K_0(1+t-t_1)^{-\frac{3}{2}}~\text{for}~t_1\geq0,\]
there holds
\begin{align*}
\mathcal{L}u^-(x,t)\leq&(-\sigma+\eta_z)z(t)-\phi'(\xi(x,t))\left[K_zz(t)+\left(\frac{|x'|^2}{\gamma t}-\alpha\right)t^{-1}\theta(x',t)-(C'+C^{K})t^{-\alpha-1}\right]\\
\leq&-\eta_zK_0t^{-\frac{3}{2}}+(\alpha\beta+C^K+C') t^{-\alpha-1}\phi'(\xi(x,t))\\
\leq&-\frac{1}{2}\eta_zt^{-\frac{3}{2}}.
\end{align*}
Indeed, since $|\xi(x,t)|$ is sufficiently large such that $(\alpha\beta+C^K+C')\phi'(\xi)\leq\frac{1}{2}\eta_zK_0$ and $\frac{1}{2}<\alpha<1$, the last inequality above holds obviously.

Case 2. We know that $\phi(\xi(x,t))\in[\eta,1-\eta]$. Since $\phi'(\xi)>0$ for all $\xi\in\mathbb{R}$, we may choose $\tau_0>0$ sufficiently small such that $\phi'(\xi)\geq\tau_0>0$. Denote $\max\limits_{[\eta,1-\eta]}f'(s)=\delta^0>0$. Then
\begin{align*}
\mathcal{L}u^-(x,t)\leq&\delta^0z(t)+\eta_zz(t)-\left[K_zz(t)+\left(\frac{|x'|^2}{t\gamma}
-\alpha\right)t^{-1}\theta(x',t)-(C^K+C')t^{-\alpha-1}\right]\phi'(\xi(x,t))\\
\leq&(-K_z\tau_0+\delta^0+\eta_z)z(t)+(\alpha\beta+C^K+C') t^{-\alpha-1}\phi'(\xi(x,t))\\
\leq&\left[-K_z\tau_0+\delta^0+\eta_z+(\alpha\beta+C^K+C')\frac{1}{K_0}\|\phi'\|_\infty\right]z(t).
\end{align*}
Let $K_z$ be sufficiently large such that  $-K_z\tau_0+\delta^0+\eta_z+(\alpha\beta+C^K+C')\frac{1}{K_0}\|\phi'\|_\infty
\leq-\frac{1}{2}\eta_z$. One hence have that $\mathcal{L}u^-(x,t)\leq-\frac{1}{2}\eta_zz(t)$. The proof is finished.
\end{proof}
\subsection{Super-solution}
\noindent

This part is devoted to the super-solution, defined as
 \[u^+(x,t)=\phi(\psi(x,t))+z(t),\]
  where
\[\psi(x,t)=x_1+c(t-1+t_\epsilon)+\theta^1(x',t)+Z(t),~\theta^1(x',t)=\beta^+t^{-\alpha^+}
e^{-\frac{|x'|^2}{t\gamma}}.\]
 Then
\[u^+_t(x,t)=(c+\theta^1_t+Z'(t))\phi'(\psi)+z'(t),\]
and
\begin{equation*}\begin{split}
&\int_\Omega J(x-y)[u^+(y,t)-u^+(x,t)]dy\\
=&\int_{\mathbb{R}^N}J(x-y)[\phi(\psi(y,t))-\phi(\psi(x,t))]dy-\int_KJ(x-y)[\phi(\psi(y,t))-\phi(\psi(x,t))]dy.
\end{split}
\end{equation*}
Note that
\[c\phi'(\psi)=\int_{\mathbb{R}^N}J(y)[\phi(\psi(x,t)-y_1)-\phi(\psi(x,t))]dy+f(\phi),\]
we have
\[u^+_t(x,t)=(\theta'_t+Z'(t))\phi'(\psi(x,t))+\int_{\mathbb{R}^N}J(y)[\phi(\psi(x,t)-y_1)
-\phi(\psi(x,t))]dy+f(\phi)+z'(t).\]
Meanwhile,
\begin{align*}
\mathcal{L}u^+(x,t)
=&(\theta'_t+Z'(t))\phi'(\psi(x,t))+\int_{\mathbb{R}^N}J(y)[\phi(\psi(x,t)-y_1)
-\phi(\psi(x,t))]dy+f(\phi)+z'(t)\\
&-\int_{\mathbb{R}^N}J(x-y)[\phi(\psi(y,t))-\phi(\psi(x,t))]dy\\
&+\int_KJ(x-y)[\phi(\psi(y,t))-\phi(\psi(x,t))]dy-f(u^+(x,t))
\end{align*}
and we obtain
\begin{align*}
\mathcal{L}u^+(x,t)
=&(\theta'_t+Z'(t))\phi'(\psi(x,t))+f(\phi)+z'(t)+\int_{\mathbb{R}^N}J(y)[\phi(\psi(x,t)-y_1)
-\phi(\psi(x-y,t))]dy\\
&+\int_KJ(x-y)[\phi(\psi(y,t))-\phi(\psi(x,t))]dy-f(u^+(x,t)).
\end{align*}
Now we focus on all the integral items above denoted by
\begin{equation*}
I:=\int_{\mathbb{R}^N}J(y)[\phi(\psi(x,t)-y_1)-\phi(\psi(x-y,t))]dy+\int_KJ(x-y)
[\phi(\psi(y,t))-\phi(\psi(x,t))]dy.
\end{equation*}
As the same progress as the calculation of $u^-(x,t)$, we have
\[I\geq-M't^{-\alpha^+-1}\phi'(\psi(x,t)).\]
Therefore,
\[\mathcal{L}u^+(x,t)\geq\left(\theta^1_t+Z'(t)-M't^{-\alpha^+-1}\right)\phi'(\psi(x,t))+z'(t)+f(\phi)-f(u^+(x,t)).\]
Next we are going to show $\mathcal{L}u^+(x,t)\geq0$ in two cases.

Case 1. Let $|\psi(x,t)|\gg1$ such that $\phi(\psi)\in[0,\eta]\cup[1-\eta,1]$, where $\eta>0$ is sufficiently small to ensure that $f'(s)\leq-\sigma<0$ for any $s\in[0,\eta]\cup[1-\eta,1]$. Since
\begin{equation*}
\left(K_z+\frac{\beta^+}{\gamma_1}\right)z(t)\geq\theta^1_t+Z'(t)-M't^{-\alpha^+-1}\geq K_zz(t)-(\alpha^+\beta^++M')t^{-\alpha^+-1}
\end{equation*}
and $|\psi(x,t)|\gg1$, we obtain that \[\left|\left(\theta^1_t+Z'(t)-M't^{-\alpha^+-1}\right)\phi'(\psi(x,t))\right|\leq\frac{1}{2}\eta_zz(t),\]
 which implies
  \[\mathcal{L}u^+(x,t)\geq-\frac{1}{2}\eta_zz(t)-\eta_zz(t)+\sigma z(t)\geq0,\]
 provided that $\eta_z<\frac{1}{2}\sigma$ and that $\frac{1}{2}\leq\alpha^+<1$.

Case 2. We have $\phi(\psi(x,t))\in[\eta,1-\eta]$. Denote $\min\limits_{s\in[0,1]}f'(s)=-\delta'<0.$  Since  $\phi'(\xi)\geq\tau_0>0$ in this case, we obtain
\begin{equation*}
\begin{split}
\mathcal{L}u^+(x,t)\geq&\left[K_zz(t)-(\alpha^+\beta^++M')t^{-\alpha^+-1}\right]\tau_0-\delta'z(t)-\eta_zz(t)\\
\geq&\left(K_z-\frac{\alpha^+\beta^++M'}{K_0}\right)\tau_0z(t)-\delta'z(t)-\eta_zz(t).
\end{split}
\end{equation*}
If $K_0t^{-\alpha^+-1}<z(t)$ and $\left(K_z-\frac{\alpha^+\beta^++M'}{K_0}\right)\tau_0-\delta'-\eta_z\geq0$, then $\mathcal{L}u^+(x,t)\geq0$. Indeed, let $\frac{1}{2}<\alpha^+<1$ and $K_z$ be sufficiently large. We finish the proof of the super-solution.
\subsection{The function $z(t)$}
\noindent

In this part, we want to construct the function $z(t)$ enlightened by Hoffman \cite{HHV}. To do this, we establish the following lemmas.
\begin{lemma}\label{l6.5}
Fix any $0<\eta_z<\ln2$. Then there exist $l_P=l_P(\eta_z)>0$ and $P_-(x)$ such that
\[P_-(-l_P)=1,\quad P_-'(-l_P)=0,\quad P_-(0)=\frac{2}{3},\quad P_-'(0)=-\frac{2}{3}\eta_z.\]
Furthermore, $P_-(x)$ satisfies
\[-\eta_zP_-(x)\leq P'_-(x)\leq0,~-l_P\leq x\leq0.\]
\end{lemma}
\begin{lemma}\label{l6.6}
Let $\eta_z$ and $l_P$ be given in Lemma \ref{l6.5} and $0<\nu\leq\eta_z$. Then there exists $P_+(x)$ such that
\[P_+(0)=\frac{2}{3},\quad P'_+(0)=-\frac{2}{3}\nu,\quad P'_+(l_P)=0.\]
Furthermore, $P_+(l_P)\geq\frac{1}{3}$ and $P_+(x)$ satisfies
\[-\eta_zP_+(x)\leq P'_+(x)\leq0,~0\leq x\leq l_P.\]
\end{lemma}
\begin{theorem}\label{l2.2}
For any $0<\eta_z<\ln2$. Then there exist constants $\mathcal{I}=\mathcal{I}(\eta_z)>0$ and $K_0=K_0(\eta_z)>0$, such that for any $t_1\geq 0$ there exists a $C^1$-smooth function $z(t):[0,+\infty)\rightarrow\mathbb{R}$ that satisfies the following properties.
\item (i) For all $t\geq0$, the inequalities $z'(t)\geq-\eta_zz(t)$  and $0<z(t)\leq z(0)=\epsilon_1$ hold.
\item (ii) In addition, $z(t)\geq K_0(1+t-t_1)^{-\frac{3}{2}}$ for all $t\geq t_1$ and
 $\int_0^{+\infty}z(t)dt<\mathcal{I}$.
\end{theorem}
\begin{proof}
First we define
\[P_-(x)=-\frac{1}{3}\eta_z^2(x+\eta_z^{-1})^2+1,~p_+(x)
=\frac{\nu\eta_z}{3}(x-\eta_z^{-1})^2+\frac{2}{3}-\frac{\nu}{3\eta_z},~0<\nu<\eta_z.\]
Let $l_P(\eta_z)=1/\eta_z$, it is easy to show that Lemmas \ref{l6.5} and \ref{l6.6} hold by a direct calculation. In addition, denote
\begin{equation*}
z_1(t)=\left\{
\begin{aligned}
  e^{-\eta_zt},\qquad\qquad\qquad\qquad & 0\leq t\leq\frac{3}{2}\eta_z^{-1}-1, \\
  \eta_z^{-\frac{3}{2}}(\frac{3}{2})^{\frac{3}{2}}e^{\eta_z-\frac{3}{2}}(1+t)^{-\frac{3}{2}}, \quad& t\geq\frac{3}{2}\eta_z^{-1}-1.
\end{aligned}
\right.
\end{equation*}
Now if $0<t_1<3\eta_z^{-1}$, then let $z(t)=\epsilon_1z_1(t)$, otherwise we define the function $z(t)$ on five different intervals. In particular, define $\nu=\frac{-z_1'(t_1-3\eta_z^{-1})}{z_1(t_1-3\eta_z^{-1})}$, which implies $0<\nu\leq\eta_z$. Then, let
\begin{equation*}z(t)=\left\{\begin{aligned}
& \epsilon_1z_1(t),\qquad\qquad\qquad\qquad\qquad\qquad\quad0\leq t\leq t_1-3\eta_z^{-1},\\
& \epsilon_1 z_1(t_1-3\eta_z^{-1})P_+(t-(t_1-3\eta_z^{-1})), ~\quad~t_1-3\eta_z^{-1}\leq t\leq t_1-2\eta_z^{-1},\\
& \epsilon_1P_-(t-t_1),\qquad\qquad\qquad\qquad\qquad\qquad ~t_1-\eta_z^{-1}\leq t\leq t_1,\\
& \epsilon_1\frac{2}{3}z_1(t-t_1), \qquad\qquad\qquad\qquad\qquad\quad ~t\geq t_1.
\end{aligned}\right.
\end{equation*}
It remains to specify $z(t)$ in $[t_1-2\eta_z^{-1}, t_1-\eta_z^{-1}]$. This can be done by choosing an arbitrary $C^1$-smooth function, under the constraints
\[z(t_1-2\eta_z^{-1})=\epsilon_1z_1(t_1-3\eta_z^{-1})P_+(\eta_z^{-1}),~z(t_1-\eta_z^{-1})=\epsilon_1\]
together with
\[z'(t_1-2\eta_z^{-1})=z'(t_1-\eta_z^{-1})=0\]
and
\[z'(t)\geq0,\quad t_1-2\eta_z^{-1}\leq t\leq t_1-\eta_z^{-1}.\]
Then we finish the proof since the properties $(i)$ and $(ii)$ are testified by a direct calculation.
\end{proof}

\begin{remark}\label{r1}{\rm
It is not difficult to see that $z(1)\geq\frac{1}{2}\epsilon_1$ from $0<\eta_z<\ln2$ and the first statement in Lemma \ref{l2.2}.}
\end{remark}

\subsection{Proofs of Theorems \ref{th1-1} and \ref{th1-4}}
\noindent

If there only exits one zero point of $f$ in $(0,1)$, then it follows from Theorems 2.4 and 2.6 of \cite{BCHV} that under the conditions $(F)$ and $(J)$, the unique solution of the stationary problem of \eqref{1-1}
\begin{equation}\label{5-2}\left\{\begin{aligned}
&\int_{\mathbb{R}^N\setminus K}J(x-y)[u(y)-u(x)]dy+f(u)=0,~x\in\mathbb{R}^N\setminus K(~\text{or}~K_{\epsilon}),\\
&0\leq u(x)\leq1,~~x\in\mathbb{R}^N\setminus K(~\text{or}~K_{\epsilon}),\\
&\sup\limits_{\mathbb{R}^N\setminus K}u(x)=1.
\end{aligned}\right.
\end{equation}
is $u = 1 ~\text{in}~\overline{\mathbb{R}^N \backslash K}(~\text{or}~\overline{\mathbb{R}^N \backslash K_{\epsilon}}).$

Now we are in position to show the main theorems. It is obvious that if the solution $u(x,t)$ satisfies the conditions in Theorem \ref{t6.1} together with Theorem \ref{t5.3}, then the conclusions in Theorem \ref{th1-1} hold. First, we know that $u(x,t)$ converges to some uniformly continuous function $V(x)$ as time tends to positive infinity. Furthermore, we claim that $V(x)$ satisfies \eqref{5-2}. Therefore, we have $V(x)\equiv1$ for all $x\in\mathbb{R}^N\backslash K$. In fact, it is sufficient to show $\sup\limits_{\mathbb{R}^N\setminus K}V(x)=1$. In view of $u(x,t)-\phi(x_1+ct)\rightarrow0$ as $t\rightarrow-\infty$ in Theorem \ref{t2.1}, for any small $\epsilon'>0$, there exist $t_{\epsilon'}$
 and $X_1>0$ sufficiently large such that
 \[
 \phi(x+ct_{\epsilon'})\geq1-\frac{\epsilon'}{2}~\text{for all}~x_1\geq X_1,
\]
and
\[|u(x,t)-\phi(x_1+ct)|\leq\frac{\epsilon'}{2}~\text{for all}~x\in\mathbb{R}^N\backslash K~\text{and}~t\leq t_{\epsilon'}.
\]
 Thus
 \[u(x,t)\geq1-\epsilon'~\text{for all}~x_1\geq X_1~\text{and}~t\geq t_\epsilon',\]
  which implies that $V(x)\geq1-\epsilon'$ due to $u(x,t)>0$. Since that  $\epsilon'$ is actually arbitrary, one has that
  \[\sup\limits_{\mathbb{R}^N\setminus K}V(x)=1.\]
   Therefore, $V(x)\equiv1$ for all $x\in\mathbb{R}^N\backslash K$. Hence, $u(x,t)\rightarrow1$ as $t\rightarrow+\infty$ for all $x\in\mathbb{R}^N\backslash K$.
 Based on the result above, the conditions in Theorem \ref{t6.1} are easy to testify following from Theorem \ref{t5.3}. Meanwhile, from $u(x,t)-\phi(x_1+ct)\rightarrow0$ as $t\rightarrow\pm\infty$ uniformly in $x\in\Omega$ and Theorem \ref{t5.3}, it is easy to obtain that $u(x,t)-\phi(x_1+ct)\rightarrow0$ as $|x|\rightarrow+\infty$ uniformly in $t\in\mathbb{R}.$ Therefore, we know that Theorem \ref{th1-1} holds. Similarly, we know the results of Theorem \ref{th1-4} hold.

\section{Appendix}
\subsection{Proof of Proposition \ref{t2.2}}
\noindent

In this subsection we intend to show the results of Proposition \ref{t2.2}. For convenience we define the operator $\mathcal{L}$ as follows
\[\mathcal{L}\omega=\omega_t-\int_\Omega J(x-y)[\omega(y,t)-\omega(x,t)]dy-f(\omega).
\]
We further show that $W^-$ is a sub-solution.
A straightforward computation shows that
\begin{equation*}\mathcal{L}W^-=\left\{
\begin{aligned}
 &-\int_\Omega J(x-y)W^-(y,t)dy ,\qquad x_1<0, \\
 &(c-\dot{\xi}(t))[\phi'(x_1+ct-\xi(t))-\phi'(-x_1+ct-\xi(t))]-\int_\Omega J(x-y)[W^-(y,t)\\
 &\quad-W^-(x,t)]dy-f(\phi(x_1+ct-\xi(t))-\phi(-x_1+ct-\xi(t))),\qquad x_1\geq0.
\end{aligned}\right.
\end{equation*}
For $x_1<0$, since $J(x)\geq0$ and $W^-\geq0$, we have
 \[\mathcal{L}W^-=-\int_\Omega J(x-y)W^-(y,t)dy\leq0\]
For $x_1\geq0$, in view of that
\begin{align*}
&\int_\Omega J(x-y)[W^-(y,t)-W^-(x,t)]dy\\
=&\int_{\mathbb{R}^N}J(x-y)[W^-(y,t)-W^-(x,t)]dy
-\int_K J(x-y)[W^-(y,t)-W^-(x,t)]dy \\
    =&\int_{\mathbb{R}^N\cap\{y_1>0\}} J(x-y)[W^-(y,t)-W^-(x,t)]dy+\int_{\mathbb{R}^N\cap\{y_1<0\}} J(x-y)[W^-(y,t)-W^-(x,t)]dy\\
&-\int_K J(x-y)[W^-(y,t)-W^-(x,t)]dy\\
\geq&\int_{\mathbb{R}^N}J(x-y)[(\phi(y_1+ct-\xi(t))
 -\phi(-y_1+ct-\xi(t)))-(\phi(x_1+ct-\xi(t))\\
 &-\phi(-x_1+ct-\xi(t)))]dy-\int_K J(x-y)[W^-(y,t)-W^-(x,t)]dy,
\end{align*}
we have
\begin{equation*}
\begin{split}
\mathcal{L}W^-\leq&-\dot{\xi}(t)[\phi'(z_+(t))-\phi'(z_-(t))] \\
    &+f(\phi(z_+(t)))-f(\phi(z_-(t)))-f(\phi(z_+(t))-\phi(z_-(t)))\\
    &+\int_KJ(x-y)[W^-(y,t)-W^-(x,t)]dy,
\end{split}
\end{equation*}
where $z_+(t)=x_1+ct-\xi(t),~z_-(t)=-x_1+ct-\xi(t)$. Assume that $K\subset\{x\in\mathbb{R}^N\mid x_1<0\}$. It follows that
\[W^-(y,t)=0 ~\text{for all}~y\in K,\]
which implies that
\begin{equation*}
\begin{split}
  \mathcal{L}W^-(x,t)\leq&-\dot{\xi}(t)[\phi'(z_+(t))-\phi'(z_-(t))] \\
  &+f(\phi(z_+(t)))-f(\phi(z_-(t)))-f(\phi(z_+(t))-\phi(z_-(t))).
\end{split}
\end{equation*}
Now we go further to show $\mathcal{L}W^-\leq0$ in two subcases.

Case A: $0<x_1<-ct+\xi(t)$.

 In this case the following lemma holds.
\begin{lemma}\label{l3.3}
Suppose that $(F)$ holds and $\phi(x_1+ct)$ satisfies \eqref{1-3} with $c>0$ and $\phi''(\xi)\geq0$ for $\xi\leq0$. Then there exists $k_3>0$ such that
\begin{equation}\label{6-1}
\phi'(\xi_1)-\phi'(\xi_2)\geq k_3[\phi(\xi_1)-\phi(\xi_2)]
\end{equation}
for $\xi_2<\xi_1<0$.
\end{lemma}
\begin{proof}
Take $M'>1$ and $(\xi_1,\xi_2)\in\mathbb{R}^2$ satisfying
\[\xi_1-M'<\xi_2<\xi_1<0.\]
Then we have
\[\phi'(\xi_1)-\phi'(\xi_2)=\phi''(\theta^1)(\xi_1-\xi_2),~\phi(\xi_1)-\phi(\xi_2)=\phi'(\theta^2)(\xi_1-\xi_2)\]
for some $(\theta^1,\theta^2)\in\mathbb{R}^2$ with $|\theta^1-\theta^2|<M'.$ Therefore, the inequality
\eqref{6-1} holds for $\xi_1-M'<\xi_2<\xi_1<0$.

If $\xi_2+M'<\xi_1<0$, then the exponential decay of $\phi(x_1)$ ensures that
\[\phi'(\xi_2)\leq\frac{1}{2}\phi'(\xi_1)\]
for $M'$ sufficiently large. This yields
\[\phi'(\xi_1)-\phi'(\xi_2)\geq\frac{1}{2}\phi'(\xi_1)\geq k_3\phi(\xi_1)\geq k_3[\phi(\xi_1)-\phi(\xi_2)].\]
The second inequality holds from the asymptotic behavior of $\phi(x_1)$. Thus we end the proof.
\end{proof}
Now we are ready to show $\mathcal{L}W^-(x,t)\leq0$. By Lemma \ref{l3.3}, we have
\begin{align*}
\mathcal{L}W^-(x,t)\leq&-\dot{\xi}(t)(\phi'(z_+(t))-\phi'(z_-(t)))+L_f\phi(z_-(t))
(\phi(z_+(t))-\phi(z_-(t)))\\
\leq&\left[-Mk_3e^{\lambda_0(ct+\xi(t))}+L_f\phi(z_-(t))\right][\phi(z_+(t))-\phi(z_-(t))]\\
\leq&\left[L_f\beta_0e^{\lambda(-x_1+ct-\xi(t))}-Mk_3e^{\lambda_0(ct-\xi(t))}\right]
[\phi(z_+(t))-\phi(z_-(t))]\\
\leq&e^{\lambda_0(ct+\xi(t))}\left[L_f\beta_0e^{(\lambda-\lambda_0)(ct+\xi(t))-2\lambda\xi(t)}
-Mk_3\right][\phi(z_+(t))-\phi(z_-(t))]\\
\leq&(L_f\beta_0-Mk_3)[\phi(z_+(t))-\phi(z_-(t))]\\
\leq&0.
\end{align*}
The last inequality holds provided that $M\geq\frac{L_f\beta_0}{k_3}$.

Case B: $x_1\geq-ct+\xi(t)$.

A direct calculation gives that
\begin{equation*}\begin{split}
\mathcal{L}W^-(x,t)\leq&-Me^{\lambda_0(ct+\xi(t))}[\phi'(z_+(t))-\phi'(z_-(t))]
+L_f\phi(z_-(t))[\phi(z_+(t))-\phi(z_-(t))]\\
\leq&L_f\phi(z_-(t))-Me^{\lambda_0(ct-\xi(t))+2\lambda_0\xi(t)}[\phi'(z_+(t))-\phi'(z_-(t))]\\
\leq&e^{-\lambda x_1+\lambda_0(ct-\xi(t))+2\lambda_0\xi(t)}\bigg[L_f\beta_0e^{(\lambda-\lambda_0)(ct-\xi(t))
-2\lambda_0\xi(t)}\\
&-M\left(\gamma_1e^{(\lambda-\mu)x_1-\mu(ct-\xi(t))}
-\delta_0e^{\lambda(ct-\xi(t))}\right)\bigg]\\
\leq&e^{-\lambda x_1+\lambda_0(ct-\xi(t))+2\lambda_0\xi(t)}\left[L_f\beta_0
-M\left(\gamma_1e^{(\lambda-\mu)x_1-\mu(ct-\xi(t))}
-\delta_0e^{\lambda(ct-\xi(t))}\right)\right].
\end{split}
\end{equation*}
If $\lambda\geq\mu$, then
\begin{equation*}\begin{split}
\mathcal{L}W^-(x,t)\leq&e^{-\lambda x_1+\lambda_0(ct-\xi(t))+2\lambda_0\xi(t)}\left[L_f\beta_0
-M\left(\gamma_1e^{-\mu(ct-\xi(t))}
-\delta_0e^{\lambda(ct-\xi(t))}\right)\right]\\
\leq&0
\end{split}
\end{equation*}
for $ct-\xi(t)\ll-1$ and $M>1$ is sufficiently large.

When $\lambda<\mu$, there holds
\begin{align*}
&f(\phi(z_+(t))))-f(\phi(z_-(t)))-f(\phi(z_+(t)))-\phi(z_-(t))))\\
=&f'(\phi(z_+(t)))\phi(z_-(t)))-o(\phi^2(z_-(t))))-f'(\phi(z_-(t))))
\phi(z_-(t)))+o(\phi^2(z_-(t))))\\
\leq&-k_4\phi(z_-(t))
\end{align*}
for $x_1+ct-\xi(t)>L_2>0$ with $L_2$ large enough, where $0<k_4<\frac{1}{2}|f'(1)-f'(0)|$. The inequality above follows from that $f'(\phi(z_+(t)))\rightarrow f'(1)$ and $f'(\phi(z_-(t)))\rightarrow f'(0)$ as $L_2\rightarrow+\infty$. Then
\begin{equation*}\begin{split}
\mathcal{L}W^-(x,t)\leq&Me^{\lambda_0(ct+\xi(t))}\phi'(z_-(t))-k_4\phi(z_-(t))\\
\leq&Me^{\lambda_0(ct+\xi(t))}\delta_0e^{\lambda(-x_1+ct-\xi(t))}-k_4\alpha_0e^{\lambda(-x_1+ct-\xi(t))}\\
\leq&e^{\lambda(-x_1+ct-\xi(t))}\left(Me^{\lambda_0(ct+\xi(t))}-k_4\alpha_0\right)\\
\leq&0,
\end{split}
\end{equation*}
provided that $ct+\xi(t)\ll-1$.

In addition, for $0<x_1+ct-\xi(t)<L_2$, there holds
\begin{equation*}\begin{split}
\mathcal{L}W^-(x,t)\leq&e^{-\lambda x_1+\lambda_0(ct-\xi(t))+2\lambda_0\xi(t)}\bigg[L_f\beta_0e^{(\lambda-\lambda_0)(ct-\xi(t))
-2\lambda_0\xi(t)}\\
&-M\left(\gamma_1e^{(\lambda-\mu)x_1-\mu(ct-\xi(t))}
-\delta_0e^{\lambda(ct-\xi(t))}\right)\bigg]\\
\leq&e^{-\lambda x_1+\lambda_0(ct-\xi(t))+2\lambda_0\xi(t)}\bigg[L_f\beta_0
-M\left(\gamma_1e^{(\lambda-\mu)L^2}e^{-\lambda(ct-\xi(t))}
-\delta_0e^{\lambda(ct-\xi(t))}\right)\bigg].
\end{split}
\end{equation*}
Since $ct-\xi(t)\ll-1$ and $M\gg1$, we have $\mathcal{L}W^-(x,t)\leq0$.

Next, we show the $W^+$ is a super-solution.
A straightforward computation shows that
\begin{equation*}\mathcal{L}W^+=\left\{
\begin{aligned}
 &2(c+\dot{\xi}(t))\phi'(x_1+ct)-f(2\phi(ct+\xi(t)))\\
 &-\int_\Omega J(x-y)[W^+(y,t)
 -W^+(x,t)]dy,~x_1<0, \\
 &(c+\dot{\xi}(t))[\phi'(x_1+ct+\xi(t))+\phi'(-x_1+ct+\xi(t))]-\int_\Omega J(x-y)[W^+(y,t)\\
 &\quad-W^+(x,t)]dy-f(\phi(x_1+ct+\xi(t))+\phi(-x_1+ct+\xi(t))),~ x_1>0.
\end{aligned}\right.
\end{equation*}
When $x_1\geq0$, denote
\[\Gamma^{+}=\{x\in\mathbb{R}^N\mid y_1>0\},~\Gamma^-=\{x\in\mathbb{R}^N\mid y_1<0\}.\]
 In view of that $K\subset\mathbb{R}^N\setminus$supp$(J)$, one gets
\begin{align*}
&\int_{\Omega}J(x-y)[W^+(y,t)-W^+(x,t)]dy\\
=&\int_{\Omega\cap\Gamma^+}J(x-y)[(\phi(y_1+ct+\xi(t))+\phi(-y_1+ct+\xi(t)))\\
&-(\phi(x_1+ct+\xi(t)+\phi(-x_1+ct+\xi(t))]dy\\
&+\int_{\Omega\cap\Gamma^-}J(x-y)[2\phi(ct+\xi(t))
-(\phi(x_1+ct+\xi(t))+\phi(-x_1+ct+\xi(t)))]dy\\
=&\int_{\mathbb{R}^N}J(x-y)[(\phi(y_1+ct+\xi(t))+\phi(-y_1+ct+\xi(t)))
-(\phi(x_1+ct+\xi(t))\\
&+\phi(-x_1+ct+\xi(t)))]dy+\int_{\Omega\cap\Gamma^-}J(x-y)[2\phi(ct+\xi(t))
-(\phi(y_1+ct+\xi(t))\\
&+\phi(-y_1+ct+\xi(t)))]dy\\
=&c(\phi'(x_1+ct+\xi(t))+\phi'(-x_1+ct+\xi(t)))-f(\phi(x_1+ct+\xi(t)))
-f(\phi(-x_1+ct+\xi(t)))\\
&+\int_{\Omega\cap\Gamma^-}J(x-y)[2\phi(ct+\xi(t))
-(\phi(y_1+ct+\xi(t))+\phi(-y_1+ct+\xi(t)))].
\end{align*}
Observe that, if $x_1>|ct+\xi(t)|>L$, where $L$ is the diameter of the compact support of $J$, then the integral item of the last equality is equal to 0. Therefore, we obtain
\begin{align*}
\mathcal{L}W^+(x,t)=&\dot{\xi}[\phi(x_1+ct+\xi(t))+\phi(-x_1+xt+\xi(t))]+f(\phi(x_1+ct+\xi(t)))\\
&+f(\phi(-x_1+ct+\xi(t)))-f(\phi(x_1+ct+\xi(t))+\phi(-x_1+ct+\xi(t)))\\
\geq&\dot{\xi}\phi(x_1+ct+\xi(t))-L_f\phi(x_1+ct+\xi(t))\phi(-x_1+ct+\xi(t))\\
\geq&e^{\lambda_0(ct+\xi(t))}\left(M\gamma_1e^{-\mu(x_1+ct+\xi(t))}-L_f\alpha_0e^{-\lambda x_1}
e^{(\lambda-\lambda_0)(ct+\xi(t))}\right).
\end{align*}

If $\mu\leq\lambda$, choosing $M\gamma_1\geq L_f\alpha_0$, it is obvious that $\mathcal{L}W^+(x,t)\geq0$ as $x_1>|ct+\xi(t)|$ is sufficiently large.

For $\mu>\lambda$, we have $f'(1)<f'(0)$. Consider the case $x_1+ct+\xi(t)\geq L_0\gg1$. Then $\phi(x_1+ct+\xi(t))\approx1$ while $\phi(-x_1+ct+\xi(t))\approx0$. Furthermore,
\begin{align*}
&f(\phi(x_1+ct+\xi(t)))+f(\phi(-x_1+ct+\xi(t)))-f(\phi(x_1+ct+\xi(t))+\phi(-x_1+ct+\xi(t)))\\
\geq&\frac{1}{2}(f'(0)-f'(1))\phi(-x_1+ct+\xi(t))\\
\geq&0,
\end{align*}
which implies that $\mathcal{L}W^+(x,t)\geq0$. For the other case $x_1+ct+\xi(t)\leq L_0$, we know
\begin{align*}
\mathcal{L}W^+(x,t)\geq&e^{\lambda_0(ct+\xi(t))}\left(M\gamma_1e^{-\mu L_0}-L_f\alpha_0e^{-\lambda x_1}
e^{(\lambda-\lambda_0)(ct+\xi(t))}\right).
\end{align*}
Since $\lambda_0<\lambda$, we obtain $\mathcal{L}W^+(x,t)\geq0$ holds if
$M\geq\frac{L_f\alpha_0}{\gamma_1}e^{\mu L_0}$.

For the case $0<x_1<|ct+\xi(t)|$, as we can see that
\begin{align*}
&\int_{\Omega\cap\Gamma^-}J(x-y)[2\phi(ct+\xi(t))
-(\phi(y_1+ct+\xi(t))+\phi(-y_1+ct+\xi(t)))]dy \\
=&\int_{\Omega\cap\{y_1<ct+\xi(t)\}}J(x-y)[2\phi(ct+\xi(t))
-(\phi(y_1+ct+\xi(t))+\phi(-y_1+ct+\xi(t)))]dy\\
&+\int_{\Omega\cap\{ct+\xi(t)<y_1<0\}}J(x-y)[2\phi(ct+\xi(t))
-(\phi(y_1+ct+\xi(t))+\phi(-y_1+ct+\xi(t)))]dy\\
:=&I_1+I_2.
\end{align*}
Since $\phi(ct+\xi(t))\leq\frac{\theta_0}{2}$ for $ct+\xi(t)\ll-1$, $\phi(0)=\theta_0,$ and $\phi'>0$, we get that $\phi(-y_1+ct+\xi(t))>\theta_0\geq2\phi(ct+\xi(t))$ for $y_1<ct+\xi(t)$. It follows that
$I_1\leq0$. We know that
\begin{equation*}
\begin{split}
I_2\leq&\int_{\Omega\cap\{ct+\xi(t)<y_1<0\}}J(x-y)C_{\phi}e^{\lambda(ct+\xi(t))}
\left(2-\left(e^{\lambda y_1}+e^{-\lambda y_1}\right)\right)dy\\
&+k_\phi e^{(k_{\phi}+\lambda)(ct+\xi(t))}\int_{\Omega\cap\{ct+\xi(t)<y_1<0\}}J(x-y)
\left(2+e^{\lambda y_1}+e^{-\lambda y_1}\right)dy\\
\leq&C_0e^{(k_{\phi}+\lambda)(ct+\xi(t))},
\end{split}
\end{equation*}
the first inequality is follows from that there exist $K_\phi,~k_\phi>0$ such that $\left|\phi(x_1)-C_{\phi}e^{\lambda x_1}\right|\leq K_{\phi}e^{(k_{\phi}+\lambda)x_1}$ for $x_1\leq0$ which is easy to obtain by \eqref{2-4}. Then we have
\begin{align*}
\mathcal{L}W^+\geq&Me^{\lambda_0(ct+\xi(t))}(\phi'(x_1+ct+\xi(t))+\phi'(-x_1+ct+\xi(t)))
+f(\phi(x_1+ct+\xi(t)))\\
&+f(\phi(-x_1+ct+\xi(t)))
-f(\phi(x_1+ct+\xi(t))+\phi(-x_1+ct+\xi(t)))-C_0e^{(k_{\phi}+\lambda)(ct+\xi(t))}\\
 \geq&Me^{\lambda_0(ct+\xi(t))}(\phi'(x_1+ct+\xi(t))+\phi'(-x_1+ct+\xi(t)))\\
 &-L_f\phi(x_1+ct+\xi(t))\phi(-x_1+ct+\xi(t))
-C_0e^{(k_{\phi}+\lambda)(ct+\xi(t))}\\
 \geq&e^{(\lambda_0+\lambda)(ct+\xi(t))}\left[2M\gamma_0-L_f\beta_0e^{(\lambda-\lambda_0)(ct+\xi(t))}
 -C_0e^{(k_{\phi}-\lambda_0)(ct+\xi(t))}\right].
\end{align*}
This gives that $\mathcal{L}W^+\geq0$, provided $2M\alpha_0>L_f\beta_0+C_0$ and $\lambda_0<\min\{k_{\phi},\lambda\}$.

For $x_1<0$, note that
\begin{align*}
&\int_\Omega J(x-y)[W^+(y,t)-W^+(x,t)]dy\\
=&\int_{\Gamma^+}J(x-y)[\phi(y_1+ct+\xi(t))+\phi(-y_1+ct+\xi(t))-2\phi(ct+\xi(t))]dy\\
\leq&\int_{-x_1}^{\infty}J_1(-y_1)[\phi(x_1+y_1+ct+\xi(t))-\phi(ct+\xi(t))]dy_1\\
&+\int_{\Gamma^+}J(x-y)[\phi(-y_1+ct+\xi(t))-\phi(ct+\xi(t))]dy\\
\leq&\int^{\infty}_{-x_1}J_1(-y_1)[\phi(y_1+ct+\xi(t))-\phi(ct+\xi(t))]dy_1\\
&+\int_{\Gamma^+}J(x-y)[\phi(-y_1+ct+\xi(t))-\phi(ct+\xi(t))]dy\\
\leq&\int_{\mathbb{R}}J_1(-y_1)[\phi(y_1+ct+\xi(t))-\phi(ct+\xi(t))]dy_1\\
&+\int_{\Gamma^+}J(x-y)[\phi(-y_1+ct+\xi(t))-\phi(ct+\xi(t))]dy\\
&-\int^{-x_1}_{-\infty}J_1(-y)[\phi(y_1+ct+\xi(t))-\phi(ct+\xi(t))]dy_1\\
\leq& c\phi'(ct+\xi(t))-f(\phi(ct+\xi(t)))
+\int_{-\infty}^{x_1}J_1(-y_1)(\phi(y_1-x_1+ct+\xi(t))-\phi(y_1+ct+\xi(t)))dy_1\\
&-\int_{x_1}^{-x_1}J_1(-y_1)(\phi(y_1+ct+\xi(t))-\phi(ct+\xi(t)))dy_1\\
=&: c\phi'(ct+\xi(t))-f(\phi(ct+\xi(t)))+II_1+II_2.\\
\end{align*}
In view of that $II_1=0$ for $x_1<-L$, we just consider the case $-L<x_1<0$. Then we have
\begin{align*}
II_1=&\int_{-\infty}^{x_1}J_1(-y)[\phi(y_1-x_1+ct+\xi(t))-\phi(y_1+ct+\xi(t))]dy\\
\leq&\int_{-\infty}^{x_1}J_1(-y)\phi'(y_1-\hat{\theta}x_1+ct+\xi(t))(-x_1)dy\\
\leq&\frac{L}{4}\phi'(ct+\xi(t)),
\end{align*}
where $\hat{\theta}\in(0,1)$, the last inequality follows from that $\phi''(\xi)>0$ for $\xi\leq0$. If we further assume that $L\leq4c$,
in addition to that
\begin{equation*}\begin{split}
II_2=&\int_{x_1}^{-x_1}J_1(-y)(\phi(ct+\xi(t))-\phi(y_1+ct+\xi(t)))dy\\
=&\int_{0}^{-x_1}J_1(-y)(2\phi(ct+\xi(t))-(\phi(y_1+ct+\xi(t))+\phi(-y_1+ct+\xi(t)))dy\\
\leq& C_0e^{(k_{\phi}+\lambda)(ct+\xi(t))},
\end{split}\end{equation*}
then,
\begin{align*}
\mathcal{L}W^+\geq&2\dot{\xi}(t)\phi'(ct+\xi(t))+f(\phi(ct+\xi(t))-f(2\phi(ct+\xi(t)))
-C_0e^{(k_{\phi}+\lambda)(ct+\xi(t))}\\
\geq&2M\gamma_0e^{(\lambda_0+\lambda)(ct+\xi(t))}-C_0e^{(k_{\phi}+\lambda)(ct+\xi(t))}\\
=&e^{(k_{\phi}+\lambda)(ct+\xi(t))}(2M\gamma_0-C_0)\\
\geq&0
\end{align*}
for $2M\geq C_0$. The second inequality follows from that $f'(s)<0$ in $[\phi(ct+\xi(t)),2\phi(ct+\xi(t))]$ for $ct+\xi(t)\ll-1$. This proves the Proposition \ref{t2.2}.

\subsection{Proof of Lemma \ref{l4.2}}
\noindent

We know that
\begin{align*}
\mathcal{M}\underline{u}:=&\underline{u}_t-\int_{\mathbb{R}^N}J(x-y)[\underline{u}(y,t)
-\underline{u}(x,t)]dy-f(\underline{u})\\
=&(c-2\epsilon\|f'\|\delta^{-1}e^{-\omega(t-t_0)})\phi'+\epsilon\omega e^{-\omega(t-t_0)}-\int_{\mathbb{R}^N}J(x-y)[\phi(\xi_-(y,t))
-\phi(\xi_-(x,t))]dy\\
&-f(\phi(\xi_-(x,t))-\epsilon e^{-\omega(t-t_0)})\\
=&-2\epsilon\|f'\|\delta^{-1}e^{-\omega(t-t_0)}\phi'+\epsilon\omega e^{-\omega(t-t_0)}+f(\phi(\xi_-(x,t)))
-f(\phi(\xi_-(x,t))-\epsilon e^{-\omega(t-t_0)}).\\
\end{align*}
When $\xi_-(x,t)\in[-A, A]$, then $\phi'(\xi_-(x,t))\geq\delta$. Therefore,
\[\mathcal{M}\underline{u}\leq\epsilon e^{-\omega(t-t_0)}(-2\|f'\|+\omega+\|f'\|)\leq0.\]
For $|\xi_-(x,t)|\geq A$, we have
\[\phi(\xi_-(x,t)),\underline{u}(x,t)\in[-\infty,\eta]\cup[1-\eta,+\infty].\]
 Then $f'(s)\leq-\omega$ for $s\in[\phi(\xi_-(xt))-\epsilon e^{-\omega(t-t_0)},\phi(\xi_-(xt))]$. Hence,
\[\mathcal{M}\underline{u}\leq\epsilon\omega e^{-\omega(t-t_0)}-\omega\epsilon e^{-\omega(t-t_0)}=0. \]
For $t_0\leq-T$, one get
\[\underline{u}(x,t_0)=\phi(x_1+ct)-\epsilon\leq u(x,t_0).\]
Until now, we have show the function $\underline{u}$ is a sub-solution to \eqref{4-2}. Similarly one can show $\overline{u}$ is a super-solution to \eqref{4-2}.

\section*{Acknowledgments}
\noindent

The second author was partially supported by NSF of China (11731005, 11671180) and the third author was partially supported by NSF of China (11601205).

\end{document}